\documentclass[12pt,reqno]{amsart}

\usepackage{amsmath,amsthm,amssymb,comment,fullpage}
\usepackage{braket}
\usepackage{mathtools}

\usepackage{caption}
\usepackage{times}
\usepackage[T1]{fontenc}
\usepackage{mathrsfs}
\usepackage{latexsym}
\usepackage[dvips]{graphics}
\usepackage{epsfig}
\usepackage{amsmath,amsfonts,amsthm,amssymb,amscd}
\input amssym.def
\input amssym.tex
\usepackage{color}
\usepackage{hyperref}
\usepackage{url}
\newcommand{\bburl}[1]{\textcolor{blue}{\url{#1}}}

\usepackage{tikz}
\usepackage{tkz-tab}
\usepackage{tkz-graph}
\usetikzlibrary{shapes.geometric,positioning}

\newcommand{\burl}[1]{\textcolor{blue}{\url{#1}}}

\numberwithin{equation}{section}

\newtheorem{thm}{Theorem}[section]

\newtheorem{defi}[thm]{Definition}

\theoremstyle{plain}

\newtheorem{corollary}[thm]{Corollary}
\newtheorem{definition}[thm]{Definition}
\newtheorem{example}[thm]{Example}
\newtheorem{lemma}[thm]{Lemma}
\newtheorem{proposition}[thm]{Proposition}
\newtheorem{theorem}[thm]{Theorem}

\newtheorem{remark}[thm]{Remark}

\newcommand{\fq}{\text{Fibonacci\ Quilt}}

\newcommand\be{\begin{equation}}
\newcommand\ee{\end{equation}}
\newcommand\bee{\begin{equation*}}
\newcommand\eee{\end{equation*}}
\newcommand\bea{\begin{eqnarray}}
\newcommand\eea{\end{eqnarray}}
\newcommand\beae{\begin{eqnarray*}}
\newcommand\eeae{\end{eqnarray*}}
\newcommand\bi{\begin{itemize}}
\newcommand\ei{\end{itemize}}
\newcommand\ben{\begin{enumerate}}
\newcommand\een{\end{enumerate}}
\newcommand\bc{\begin{center}}
\newcommand\ec{\end{center}}
\newcommand\ba{\begin{array}}
\newcommand\ea{\end{array}}

\newcommand{\kmin}[1]{k_{\min}(#1)}
\newcommand{\kmax}[1]{k_{\max}(#1)}

\newcommand{\B}{\mathcal{B}}

\newcommand{\Z}{\ensuremath{\mathbb{Z}}}

\newcommand\frakfamily{\usefont{U}{yfrak}{m}{n}}
\DeclareTextFontCommand{\textfrak}{\frakfamily}

\newcommand{\cm}{c_{\text{{\rm mean}}}}
\newcommand{\cv}{c_{\text{{\rm variance}}}}

\newtheorem{rek}[thm]{Remark}

\newcommand{\hr}[1]{\href{#1}{\url{#1}}}

\newcommand{\dfq}{d_{\rm FQ}}
\newcommand{\dave}{d_{\rm FQ; ave}}

\newcommand{\PP}[1]{\mathbb{P}[#1]}
\newcommand{\E}[1]{\mathbb{E}[#1]}
\newcommand{\V}[1]{\text{{\rm Var}}[#1]}
\newcommand{\BS}{\mathcal{S}}
\newcommand{\T}{\mathcal{T}}
\newcommand{\LT}{\mathcal{L_T}}
\newcommand{\LS}{\mathcal{L_S}}
\newcommand{\ZS}{\mathcal{Z_S}}

\newcommand{\nsum}{Y_n}

\title{New Behavior in Legal Decompositions Arising from Non-positive Linear Recurrences}

\author{Minerva Catral}
\email{\textcolor{blue}{\href{mailto:catralm@xavier.edu}{catralm@xavier.edu}}}
\address{Department of Mathematics, Xavier University, Cincinnati, OH 45207}

\author{Pari L. Ford}
\email{\textcolor{blue}{\href{mailto:fordpl@bethanylb.edu}{fordpl@bethanylb.edu}}}
\address{Department of Mathematics and Physics, Bethany College, Lindsborg, KS 67456}

\author{Pamela E. Harris}
\email{\textcolor{blue}{\href{mailto:Pamela.Harris@usma.edu}{Pamela.Harris@usma.edu}}}
\address{Department of Mathematical Sciences, United States Military Academy, West Point, NY 10996}

\author{Steven J. Miller}
\email{\textcolor{blue}{\href{mailto:sjm1@williams.edu}{sjm1@williams.edu}},  \textcolor{blue}{\href{Steven.Miller.MC.96@aya.yale.edu}{Steven.Miller.MC.96@aya.yale.edu}}}
\address{Department of Mathematics and Statistics, Williams College, Williamstown, MA 01267}

\author{Dawn Nelson}
\email{\textcolor{blue}{\href{mailto:dnelson1@saintpeters.edu}{dnelson1@saintpeters.edu}}}
\address{Department of Mathematics, Saint Peter's University, Jersey City, NJ 07306}

\author{Zhao Pan}
\email{\textcolor{blue}{\href{mailto:zhaop@andrew.cmu.edu}{zhaop@andrew.cmu.edu}}}
\address{Department of Mathematics, Carnegie Mellon University, Pittsburgh, PA 15213}

\author{Huanzhong Xu}
\email{\textcolor{blue}{\href{mailto:huanzhox@andrew.cmu.edu}{huanzhox@andrew.cmu.edu}}}
\address{Department of Mathematics, Carnegie Mellon University, Pittsburgh, PA 15213}

\thanks{The fourth named author was partially supported by NSF grants DMS1265673 and DMS1561945 and Carnegie Mellon University. This research was performed while the third named author held a National Research Council Research Associateship Award at USMA/ARL. This work was begun at the 2014 REUF Meeting at AIM; it is a pleasure to thank AIM and ICERM for  their support, and the participants there and at the 16\textsuperscript{th} and 17\textsuperscript{th} International Conference on Fibonacci Numbers and their Applications for helpful discussions.}



\date{\today}

\begin{document}

\maketitle

\begin{abstract} Zeckendorf's theorem states every positive integer has a unique decomposition as a sum of non-adjacent Fibonacci numbers. This result has been generalized to many sequences $\{a_n\}$ arising from an integer positive linear recurrence, each of which has a corresponding notion of a legal decomposition. Previous work proved the number of summands in decompositions of $m \in [a_n, a_{n+1})$ becomes normally distributed as $n\to\infty$, and the individual gap measures associated to each $m$ converge to geometric random variables, when the leading coefficient in the recurrence is positive. We explore what happens when this assumption is removed in two special sequences. In one we regain all previous results, including unique decomposition; in the other the number of legal decompositions exponentially grows and the natural choice for the legal decomposition (the greedy algorithm) only works approximately 92.6\% of the time (though a slight modification always works). We find a connection between the two sequences, which explains why the distribution of the number of summands and gaps between summands behave the same in the two examples. In the course of our investigations we found a new perspective on dealing with roots of polynomials associated to the characteristic polynomials. This  allows us to remove the need for the detailed technical analysis of their properties which greatly complicated the proofs of many earlier results in the subject, as well as handle new cases beyond the reach of existing techniques.
\end{abstract}

\tableofcontents


\section{Introduction}



Previous work on Positive Linear Recurrence Sequences (PLRS) generalized Zeckendorf's theorem, which states that every positive integer can be uniquely written as a sum of nonconsecutive Fibonacci numbers. Papers such as \cite{MW1,MW2,DDKMMV,DDKMV} showed that the decompositions of positive integers as sums of elements from a PLRS are unique and that the average number of summands displays Gaussian behavior; see also \cite{Day, DG, FGNPT, GT, GTNP, Ha, Ho, Ke, LT, Len, Lek, Ste1, Ste2, Ze}, and see \cite{Al, DDKMMV, DDKMV} for other types of decomposition laws. Subsequent papers \cite{BBGILMT,BILMT} included proofs of the exponential decay in the gaps between summands. These papers hinge on technical arguments depending on the leading term of the recurrence relation defining the sequence being non-zero.

We have two goals in the work below: (1) we develop a new combinatorial method to bypass the technical arguments on polynomials associated to the recurrence relation which complicated arguments in previous work, and (2) we explore the behavior of some special integer sequences satisfying recurrences with leading term zero. The second is particularly interesting as all of the previous results are not applicable, and we have to develop new methods. Interestingly, while the two new sequences we introduce at first seem unrelated, knowledge of the first yields many results for the second (and thus explains why we study these two together).

The first new infinite two-parameter family of sequences are called the $(s,b)$-Generacci sequences. They were introduced in \cite{CFHMN1}, where we showed that the $(1,2)$-Generacci sequence, also referred to as the Kentucky sequence, has similar behavior to those displayed by a PLRS even though it is not a PLRS (the $(1,1)$ case is the Fibonacci numbers, hence the name). This included the Gaussian behavior for the number of summands and the exponential decay in gaps between summands \cite[Theorems 1.5 and 1.6]{CFHMN1}. In \cite{CFHMN2}, we further expanded the study of the $(s,b)$-Generacci sequences and proved that these sequences lead to unique decompositions of all positive integers. In this paper, we introduce new methods which lead to proofs of Gaussian behavior in the number of summands, both for this sequence and others in the literature, which allow us to avoid complications involving roots of polynomials. This is very much in contrast to the very technical arguments presented in \cite{MW1} for Positive Linear Recurrences. In addition, we provide an analogous result on the geometric decay in the distribution of gaps between bins (the arguments for gaps between summands is similar but involves uninteresting additional book-keeping, and hence we omit them here).

The other sequence of interest is called the Fibonacci Quilt sequence. This sequence arises naturally from a 2-dimensional construction of a log-cabin style quilt. The Fibonacci Quilt sequence, like the $(s,b)$-Generacci sequences, satisfies a recurrence with leading term zero, however in \cite{CFHMN2} we showed that the legal decompositions arising from this sequence have drastically different behavior than that of the $(s,b)$-Generacci sequence, with the major difference being that the decompositions arising from the Fibonacci Quilt sequence are not unique. In fact, we showed that the number of legal decompositions of a positive integer grows exponentially as the integer increases. Another surprising result is that among all of these decompositions, the decomposition arising from the greedy algorithm is a legal decomposition (approximately) 93\% of the time. In \cite{CFHMN2}, we defined a modified greedy algorithm, called the \emph{Greedy-6 algorithm}, and showed that the decomposition arising from this algorithm always terminates in a legal decomposition. Moreover, we showed that the Greedy-6 algorithm results in a legal decomposition with minimal number of summands. Interestingly, while there is markedly different behavior between these two new sequences in terms of uniqueness of decompositions, they exhibit similar behavior in terms of the number of summands and gaps between summands. In particular, for the Greedy-6 decomposition we obtain Gaussian behavior for the number of summands and geometric decay for the average and individual gap measures almost immediately by noticing a connection between the Fibonacci Quilt and $(4,1)$-Generacci sequences.

Below we describe the sequences in greater detail and then state our main results. In the companion paper \cite{CFHMN2} we have collected many of the basic properties of the sequences we study; we repeat the statements here so this paper may be read independently of \cite{CFHMN2}. As many of the calculations follow analogously to similar computations in the literature, we only provide the details for the new arguments; the more standard proofs are available in the appendices of this paper.



\subsection{$(s,b)$-Generacci Sequences and the Fibonacci Quilt Sequence}


\subsubsection{$(s,b)$-Generacci Sequences} \ \\
We begin by restating the definition and some computational results for the $(s,b)$-Generacci sequences. The proofs of these results appeared in \cite{CFHMN2}, and follow from straightforward algebra applied to the definitions.

Briefly, the sequence is defined as follows. We have a collection of bins $\mathcal{B}_j$, each containing $b$ numbers. We construct a sequence $\{a_n\}$ such that each positive integer has a  decomposition as a sum of elements such that (1) we take at most one element in a bin, and (2) if we take an element in bin $\mathcal{B}_j$, then we do not take any elements in any of the  $s$ bins preceding $\mathcal{B}_j$ nor the $s$ bins succeeding $\mathcal{B}_j$. We formalize the above in the following two definitions.

\begin{defi}[($s,b$)-Generacci legal decomposition]
For fixed integers $s, b \geq 1$, let an increasing sequence of positive integers $\{a_i\}_{i=1}^\infty$ and a family of subsequences $\mathcal{B}_n=\{a_{b(n-1)+1},\ldots,a_{bn}\}$ be given (we call these subsequences {\em bins}). We declare a decomposition of an integer $m = a_{\ell_1} + a_{\ell_2} + \dots + a_{\ell_k}$  where $a_{\ell_i} > a_{\ell_{i+1}}$  to be an {\em $(s,b)$-Generacci  legal decomposition} provided $\{a_{\ell_i}, a_{\ell_{i+1}}\} \not\subset \mathcal{B}_{j-s}\cup \mathcal{B}_{j-s+1}\cup \dots \cup \mathcal{B}_{j}$ for all $i,j$, with the convention that $\mathcal{B}_{j} = \emptyset$ for $j\leq 0$.
\end{defi}

\begin{defi}[($s,b$)-Generacci sequence]\label{sbDefi}
For fixed integers $s, b \geq 1$, an increasing sequence of positive integers $\{a_i\}_{i=1}^\infty$ is the {\em $(s,b)$-Generacci sequence} if every $a_i$ for $i \geq 1$ is the smallest positive integer that does not have an $(s,b)$-Generacci legal decomposition using the elements $\{a_1, \dots, a_{i-1}\}$.
\end{defi}

We recall that Zeckendorf's theorem gave an equivalent definition of the Fibonacci numbers as the unique sequence which allows one to write all positive integers as a sum of nonconsecutive elements in the sequence. Note this holds provided we define the Fibonacci numbers beginning with $1,2,3,\ldots$. It is then clear that the $(1,1)$-Generacci sequence is the Fibonacci sequence.  However, other interesting sequences are also  $(s,b)$-Generacci  sequences. For example, Narayana's cow sequence is the $(2,1)$-Generacci sequence and the Kentucky sequence (studied at length by the authors in \cite{CFHMN1}) is the $(1,2)$-Generacci sequence.

\begin{theorem}[Recurrence Relation and Explicit Formula]\label{thrm:recurrencesb_2}
For $n > (s+1)b+1$,  the $n^{\text{th}}$ term of the $(s,b)$-Generacci sequence satisfies
\begin{equation}\label{explicitConstant_2}
a_n \ = \ a_{n-b}+ba_{n-(s+1)b} \ = \  c_1\lambda_1^n \left[1 + O\left((\lambda_2/\lambda_1)^n\right)\right],
\end{equation}
where $\lambda_1$ is the largest root of $x^{(s+1)b} - x^{sb} - b  = 0$, and $c_1$ and $\lambda_2$ are constants with $\lambda_1>1$,   $c_1 > 0 $ and $|\lambda_2| < \lambda_1$.
\end{theorem}

The proof of the recurrence follows from standard arguments involving the construction of the  $(s,b)$-sequence (see, e.g., \cite[Theorem 1.3]{CFHMN2}). The proof of the main term and error bound follows from a generalized Binet formula (see, e.g., \cite[Theorem A.1]{BBGILMT}) and we provide a proof in \S 2.1 of \cite{CFHMN2}.  There is a slight complication in that the leading coefficient of the recurrence is zero; we surmount this by passing to a related recurrence where the leading coefficient is positive and thus the standard arguments apply.

\subsubsection{Fibonacci Quilt Sequence} \ \\

We state the definition and some computational results for the Fibonacci Quilt sequence; the proofs follow immediately by straightforward algebra (see \cite{CFHMN2}). Unlike many other works in the subject, here we use the more common convention for the Fibonacci numbers that $F_0 = 1$, $F_1 = 1$ (and of course still taking $F_{n+1} = F_n + F_{n-1})$. With this notation an interesting property of the Fibonacci numbers is that they can be used  to tile the plane by squares (see Figure \ref{fig:fib_spiral}).

\begin{figure}[h]
\begin{center}
\scalebox{.4}{\includegraphics{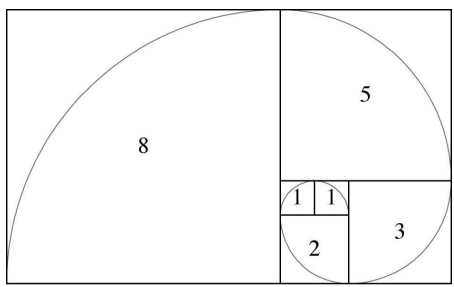}}
\caption{\label{fig:fib_spiral} The (start of the) Fibonacci Spiral.}
\end{center}\end{figure}

We have a new notion of legality based on the spiral and motivated by the Zeckendorf rule for the Fibonacci numbers involving the use of non-adjacent terms. We create a new sequence of integers by placing the integers of the sequence in the squares of the spiral (in the order the spiral is created) using the rule that we must be able to decompose every positive integer as a sum of elements in the sequence provided the squares they lie in do not share part of a side.

\begin{definition}[FQ-legal decomposition]
Let an increasing sequence of positive integers $\{q_i\}_{i=1}^\infty$ be given. We declare a decomposition of an integer
\begin{align}m \ = \  q_{\ell_1} + q_{\ell_2} +\cdots + q_{\ell_t}\end{align}
(where $q_{\ell_i} > q_{\ell_{i+1}}$) to be an FQ-legal decomposition if for all $i,j$, $|\ell_i-\ell_j|\neq 0,1,3,4$ and $\{1,3\}\not\subset\{\ell_1,\ell_2,\ldots \ell_t\}$.
\end{definition}

We compress the Fibonacci spiral so that the $n$\textsuperscript{th} square is replaced with a rectangle of thickness 1 (this allows us to display more of the pattern in the same space); we call this the Fibonacci Quilt (see Figure \ref{Fig:FibQuilt}). The adjacency of the squares in the Fibonacci spiral is identical to the adjacency of the rectangles in the Fibonacci Quilt. (The latter figure is known in the quilting community as the log cabin quilt pattern, and we adopt the name Fibonacci Quilt sequence from this connection.) The definition above states that we cannot use two terms if the rectangles they are placed in share part of an edge. We see that  $q_n+q_{n-1}$ is not legal but $q_n+q_{n-2}$ is legal for $n \ge 4$. For small $n$, the starting pattern of the quilt forbids decompositions that contain $q_3+q_1$.
\medskip

\begin{minipage}{\textwidth}
\begin{minipage}[b]{0.5\textwidth}
\centering
\resizebox{.85\textwidth}{!}{
\begin{tikzpicture}
\draw (0,0) rectangle (1,1);
\node at (.5,.5) {$q_1$};
\draw (0,0) rectangle (1,-1);
\node at (.5,-.5) {$q_2$};
\draw (1,-1) rectangle (2,1);
\node at (1.5,0) {$q_3$};
\draw (2,1) rectangle (0,2);
\node at (1,1.5) {$q_4$};
\draw (0,2) rectangle (-1,-1);
\node at (-.5,.5) {$q_5$};
\draw (-1,-1) rectangle (2,-2);
\node at (.5,-1.5) {$q_6$};
\draw (2,-2) rectangle (3,2);
\node at (2.5,0) {$q_7$};
\draw (3,2) rectangle (-1,3);
\node at (1,2.5) {$\vdots$};
\draw (-1,3) rectangle (-2,-2);
\node at (-1.5,.5) {$\cdots$};
\draw (-2,-2) rectangle (3,-3);
\node at (.5,-2.5) {$\vdots$};
\draw (3,-3) rectangle (4,3);
\node at (3.5,0) {$\cdots$};
\draw (4,3) rectangle (-2,4);
\node at (1,3.5) {$q_{n-4}$};
\draw (-2,4) rectangle (-3,-3);
\node at (-2.5,.5) {$q_{n-3}$};
\draw (-3,-3) rectangle (4,-4);
\node at (.5,-3.5) {$q_{n-2}$};
\draw (4,-4) rectangle (5,4);
\node at (4.5,0) {$q_{n-1}$};
\draw (5,4) rectangle (-3,5);
\node at (1,4.5) {$q_{n}$};
\draw (-3,5) rectangle (-4,-4);
\node at (-3.5,.5) {$q_{n+1}$};
\draw (-4,-4) rectangle (5,-5);
\node at (.5,-4.5) {$q_{n+2}$};
\draw (5,-5) rectangle (6,5);
\node at (5.5,0) {$q_{n+3}$};
\end{tikzpicture}}

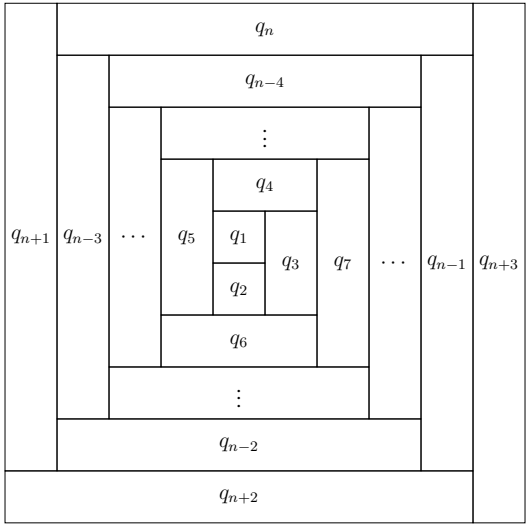
\captionof{figure}{Log Cabin Quilt Pattern}\label{Fig:LogCabin}
 \end{minipage}
  \begin{minipage}[b]{0.5\textwidth}
\centering
\resizebox{.85\textwidth}{!}{
\begin{tikzpicture}
\draw (0,0) rectangle (1,1);
\node at (.5,.5) {$1$};
\draw (0,0) rectangle (1,-1);
\node at (.5,-.5) {$2$};
\draw (1,-1) rectangle (2,1);
\node at (1.5,0) {$3$};
\draw (2,1) rectangle (0,2);
\node at (1,1.5) {$4$};
\draw (0,2) rectangle (-1,-1);
\node at (-.5,.5) {$5$};
\draw (-1,-1) rectangle (2,-2);
\node at (.5,-1.5) {$7$};
\draw (2,-2) rectangle (3,2);
\node at (2.5,0) {$9$};
\draw (3,2) rectangle (-1,3);
\node at (1,2.5) {$12$};
\draw (-1,3) rectangle (-2,-2);
\node at (-1.5,.5) {$16$};
\draw (-2,-2) rectangle (3,-3);
\node at (.5,-2.5) {$21$};
\draw (3,-3) rectangle (4,3);
\node at (3.5,0) {$28$};
\draw (4,3) rectangle (-2,4);
\node at (1,3.5) {$37$};
\draw (-2,4) rectangle (-3,-3);
\node at (-2.5,.5) {$49$};
\draw (-3,-3) rectangle (4,-4);
\node at (.5,-3.5) {$65$};
\draw (4,-4) rectangle (5,4);
\node at (4.5,0) {$86$};
\draw (5,4) rectangle (-3,5);
\node at (1,4.5) {$114$};
\draw (-3,5) rectangle (-4,-4);
\node at (-3.5,.5) {$151$};
\draw (-4,-4) rectangle (5,-5);
\node at (.5,-4.5) {$200$};
\draw (5,-5) rectangle (6,5);
\node at (5.5,0) {$265$};
\draw (6,5) rectangle (-4,6);
\node at (1,5.5) {$351$};
\draw (-4,6) rectangle (-5,-5);
\node at (-4.5,.5) {$465$};
\end{tikzpicture}}

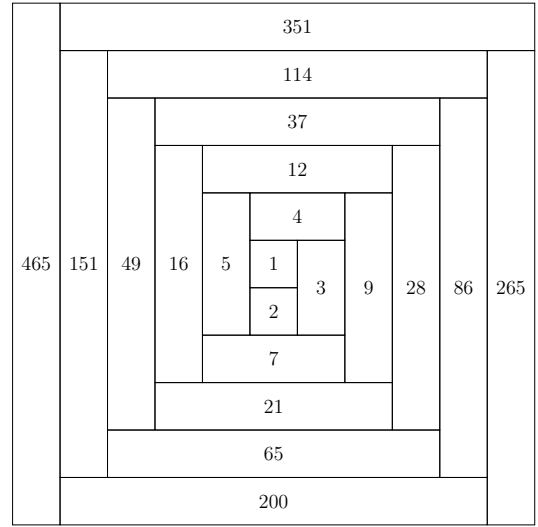
\captionof{figure}{Fibonacci Quilt Sequence}\label{Fig:FibQuilt}
\end{minipage}
  \end{minipage}


\medskip

The discussion above motivates the following definition of the Fibonacci Quilt Sequence.

\begin{definition}[Fibonacci Quilt Sequence]
The Fibonacci Quilt Sequence $\{q_i\}_{i=1}^\infty$ has $q_1 = 1$ and every $q_i$ ($i\geq 2$) is the smallest positive integer that does not have an FQ-legal decomposition using the elements $\{q_1,\ldots,q_{i-1}\}$.
\end{definition}

We display the first few terms of this sequence in Figure \ref{Fig:FibQuilt}: $\{1, 2, 3, 4, 5, 7, 9, 12, \dots\}$.

\begin{thm}[Recurrence Relations]\label{thm:rrfibquilt}

Let $q_n$ denote the $n$\textsuperscript{th} term in the Fibonacci Quilt Sequence. Then
(1) for $n\geq 6$, $q_{n+1} = q_{n}  + q_{n-4}$, (2) for $n \geq 5$, $q_{n+1} = q_{n-1}  + q_{n-2}$, and (3) we have \be q_n \ = \ \alpha_1 \lambda_1^n + \alpha_2 \lambda_2^n + \alpha_3 \overline{\lambda_2}^n, \ee where $\alpha_1 \approx 1.26724 $, \be \lambda_1 \ = \ \frac13 \left(\frac{27}{2} - \frac{3 \sqrt{69}}{2}\right)^{1/3} + \frac{\left(\frac12 \left(9 + \sqrt{69}\right)\right)^{1/3}}{3^{2/3}} \ \approx \ 1.32472 \ee and $\lambda_2 \approx -0.662359 - 0.56228i$ (which has absolute value approximately 0.8688).

\end{thm}

The above result appeared in \cite[Theorem 1.6 and Proposition 2.4]{CFHMN2}, and follows from a straightforward constructive proof using induction.


%

\subsection{Results}

Both the $(s,b)$-Generacci sequences and the Fibonacci quilt sequence satisfy recurrence relations with leading term zero. They display drastically different behavior in some respects, but also have very similar behavior for other problems (which allows us to deduce results for the Fibonacci Quilt sequence from results for the $(4,1)$-Generacci sequence). We begin by stating results related to the decompositions arising from these sequences, many of which are proved in the companion paper \cite{CFHMN2}. We then state new results on Gaussian behavior in the number of summands, and exponential decay in the gap measures between summands.

\subsubsection{Decompositions}

The $(s,b)$-Generacci legal decompositions are unique (\cite[Theorem 1.9]{CFHMN2}) whereas FQ-legal decompositions are not. The average number of FQ-legal decompositions grows exponentially \cite[Theorem 1.11]{CFHMN2}.

Let $m$ be a positive integer and let $\dfq(m)$ denote the number of FQ-legal decompositions of $m$. Let $\dave(n)$ denote the average number of FQ-legal decompositions of integers in $I_n := [0,q_{n+1})$. Hence \be \dave(n) \ := \ \frac1{q_{n+1}} \sum_{m=0}^{q_{n+1}-1} \dfq(m). \ee

\begin{thm}[Growth Rate of Average Number of Decompositions]\label{thm:growthratenumberdecomp} 
There exist computable constants $\lambda\approx 1.05459$ and $C_2 > C_1 > 0$ such that for all $n$ sufficiently large, \be C_1 \lambda^n \ \le \ \dave(n)\ \le\ C_2 \lambda^n.\ee Thus the average number of FQ-legal decompositions of integers in $[0, q_{n+1})$ tends to infinity exponentially fast.
\end{thm}

The proof of Theorem \ref{thm:growthratenumberdecomp} (found in \cite{CFHMN2}) derived recurrence relations and an explicit formula for the number of FQ-legal decompositions.

In many decomposition schemes including the ($s,b$)-Generacci case, there is a unique legal representation which can be found through a greedy algorithm. For the Fibonacci Quilt, not only does uniqueness often fail, but frequently the greedy algorithm does not terminate in a FQ-legal decomposition. For example, if we try to decompose $6 \in [q_5,q_6),$ the greedy algorithm would start with the largest summand possible, $q_5 =5$. Unfortunately at this point we would need to take $q_1 = 1$ as our next term, but we cannot as $q_1$ and $q_5 $ share a side. The only decomposition of 6 bypasses $q_5$ and uses $q_4$, writing it as $q_4 + q_2$. In \cite[Theorem 1.13]{CFHMN2}, we determined how often the greedy algorithm yields a legal decomposition.

\begin{thm}\label{thm:successgreedyalg} There is a constant $\rho \in (0,1)$ such that, as $n\to\infty$, the percentage of positive integers in $[1,q_n)$ where the greedy algorithm terminates in a Fibonacci Quilt legal decomposition converges to $\rho$. This constant is  approximately 0.92627.
\end{thm}

The proof of Theorem \ref{thm:successgreedyalg} (found in \cite{CFHMN2}) used a recurrence for $h_n$ which denotes the number of positive integers between $1$ and $q_{n+1}-1$ where the greedy algorithm successfully terminates in a legal decomposition. The result then follows from the recurrence and the use of a generalized Binet formula.

Even though Theorem \ref{thm:successgreedyalg} shows that the greedy algorithm does not always terminate in a FQ-legal decomposition, a simple modification \emph{does} always terminate in a FQ-legal decomposition. The Greedy-6 Algorithm (defined in Definition \ref{alg:greedy6}) is identical to the greedy algorithm with the caveat that if the greedy algorithm yields a decomposition including $q_1$ and $q_5$ (which sum to $6$) we exchange them with the summands $q_2$ and $q_4$ (also summing to $6$).


\begin{definition}{(Greedy-6 Algorithm)}\label{alg:greedy6}
Decompose $m$ into sums of FQ-numbers as follows.

\begin{itemize}
\item If there is an $n$ with $m=q_n$ then we are done.
\item If $m = 6,$ then we decompose $m$ as $q_4+q_2$ and we are done.
\item If $m \geq q_6$ and $m \neq q_n$ for all $n \geq 1$, then we write $m=q_{\ell_1} + x$ where $q_{\ell_1}< m< q_{\ell_1+1}$ and $x > 0$.  We then iterate the process with $m:=x$.
\end{itemize}

We denote the decomposition of $m$ that results from the Greedy-6 Algorithm by $\mathcal{G}(m)$.
\end{definition}

\begin{theorem}\label{thm:greedy6}
For all $m > 0,$ the Greedy-6 Algorithm results in a FQ-legal decomposition.  Moreover, if $\mathcal{G}(m) = q_{\ell_1}+q_{\ell_2}+\dots +q_{\ell_{t-1}}+q_{\ell_t}$ with $q_{\ell_1}>q_{\ell_2}>\cdots>q_{\ell_t}$,  then the decomposition satisfies exactly one of the following conditions:
\begin{enumerate}
\item $\ell_i-\ell_{i+1}\geq 5$ for all $i$ or
\item $\ell_i-\ell_{i+1}\geq 5$ for $i\leq t-3$ and $\ell_{t-2}\geq10,\, \ell_{t-1}=4, \,\ell_{t}=2$.
\end{enumerate}

Further, if $m=q_{\ell_1}+q_{\ell_2}+\cdots+q_{\ell_{t-1}}+q_{\ell_t}$ with $q_{\ell_1}>q_{\ell_2}>\cdots>q_{\ell_t}$  denotes a decomposition of $m$ where either
\begin{enumerate}
\item $\ell_i-\ell_{i+1}\geq 5$ for all $i$ or
\item $\ell_i-\ell_{i+1}\geq 5$ for $i\leq t-3$ and $\ell_{t-2}\geq10,\, \ell_{t-1}=4, \,\ell_{t}=2$,
\end{enumerate}
then $q_{\ell_1}+q_{\ell_2}+\cdots+q_{\ell_{t-1}}+q_{\ell_t}=\mathcal{G}(m)$. That is, the decomposition of $m$ is the Greedy-6 decomposition.
\end{theorem}

The proof is straightforward; see \cite[Theorem 1.15]{CFHMN2}.

Let $\mathcal{D}(m)$ be a given decomposition of $m$ as a sum of Fibonacci Quilt numbers (not necessarily legal): \be m \ = \ c_1 q_1 + c_2 q_2 + \cdots + c_n q_n, \ \ \ c_i \in \{0, 1, 2, \dots\}. \ee We define the
number of summands by \be \#{\rm summands}(\mathcal{D}(m)) \ := \ c_1 + c_2 + \cdots + c_n. \ee 
We can now state our final result for the Fibonacci Quilt sequence and the number of summands in FQ-legal decompositions; the proof is again standard and given in \cite[Theorem 1.16]{CFHMN2}.

\begin{thm}\label{thm:Dm} If $\mathcal{D}(m)$ is any decomposition of $m$ as a sum of Fibonacci Quilt numbers, then \be \#{\rm summands}(\mathcal{G}(m)) \ \le \ \#{\rm summands}(\mathcal{D}(m)).\ee
\end{thm}


\subsubsection{Gaussian Distribution of  the Number of Summands}

One of our main theorems regarding the  $(s,b)$-Generacci sequences states that the number of summands in the $(s,b)$-Generacci legal decompositions of the positive integers follow a Gaussian distribution. We reiterate that previous results for Positive Linear Recurrences do not apply since the $(s,b)$-Generacci sequences are not Positive Linear Recurrences. Moreover, previous proofs of Gaussian behavior were very technical and relied heavily on knowledge of roots of polynomials. In this paper,  some of the ideas we use are similar to those employed when studying Positive Linear Recurrence sequences but there is a major difference. We present a new technique that allows us to bypass all of the technical assumptions required in the other papers in their proofs of Gaussianity; see also \cite{B-AM, LiM} for two different approaches (the first using Markov chains, the second using two dimensional recurrences) which also successfully avoid these complications. In \S\ref{sec:gaussian} we give a proof to this main result, and then show it is applicable to the two sequences of this paper.


\begin{theorem}[Gaussian Behavior of Summands for $(s,b)$-Generacci]\label{thm:gaussian}
Let the random variable $Y_n$ denote the number of summands in the (unique) $(s,b)$-Generacci legal decomposition of an integer chosen uniformly at random from $[a_{(n-1)b+1}, a_{nb+1})$. Normalize $Y_n$ to $Y_n' = (Y_n - \mu_n)/\sigma_n$, where $\mu_n$ and $\sigma_n$ are the mean and variance of $Y_n$ respectively, which satisfy
\begin{align}\label{muConstantA} \mu_n  \ = \   An+B+o(1), \ \ \ \ \sigma_n^2  \ = \   Cn+D+o(1),
 \end{align} for some positive constants $A,B,C,D$. Then $Y_n'$ converges in distribution to the standard normal distribution as $n \rightarrow \infty$.
\end{theorem}

\begin{remark} Using the methods of \cite{BDEMMTTW}, these results can trivially be extended to hold for an integer chosen uniformly at random from $[1, a_{nb+1})$ by trivially combining the results for intervals of the form $[a_{\ell b +1}, a_{(\ell+1)b +1})$.
\end{remark}

By specializing the above to the $(4,1)$-Generacci sequence we immediately obtain the same result for the Greedy-6 decompositions of the Fibonacci Quilt.

\begin{theorem}[Gaussian Behavior of Summands for Greedy-6 FQ-Legal Decompositions]\label{thm:gaussianFQ}
Let the random variable $Y_n$ denote the number of summands in the (unique) Greedy-6 FQ-legal decomposition of an integer chosen uniformly at random from $[q_n, q_{n+1})$.\footnote{Using the methods of \cite{BDEMMTTW}, these results can be extended to hold almost surely for sufficiently large sub-interval of $[q_{n}, q_{n+1})$. 
} Normalize $Y_n$ to $Y_n' = (Y_n - \mu_n)/\sigma_n$, where $\mu_n$ and $\sigma_n$ are the mean and variance of $Y_n$ respectively, which satisfy
\begin{align} \mu_n \ = \   \widetilde{A}n+\widetilde{B}+o(1), \ \ \ \ \sigma_n^2  \ = \   \widetilde{C}n+\widetilde{D}+o(1),
 \end{align} for some positive constants $\widetilde{A},\widetilde{B},\widetilde{C},\widetilde{D}$. Then $Y_n'$ converges in distribution to the standard normal distribution as $n \rightarrow \infty$.
\end{theorem}


\subsubsection{Gaps between Summands}\label{sbgaps}

The following results concern the behavior of gaps between bins for $(s,b)$-Generacci sequences. For  $m \in [a_{(n-1)b+1},a_{nb+1})$, the legal decomposition
\begin{align}
m\ = \ a_{\ell_1} + a_{\ell_2}  + \cdots + a_{\ell_k} \ \ \ {\rm with} \ \ \ \ell_1\ >\ \ell_2 \ > \  \cdots \ > \  \ell_k, \end{align}
where $a_{\ell_i}\in \mathcal{B}_{\left\lceil\frac{\ell_i}{b}\right\rceil}$ for all $1\leq i\leq k$,  we define the  set of bin gaps as follows:
\begin{align}
\text{BGaps}(m) \ := \  \left\{\left\lceil \frac{\ell_1}{b}\right\rceil-\left\lceil \frac{\ell_2}{b}\right\rceil, \left\lceil \frac{\ell_2}{b}\right\rceil - \left\lceil \frac{\ell_3}{b}\right\rceil, \dots, \left\lceil \frac{\ell_{k-1}}{b}\right\rceil - \left\lceil \frac{\ell_{k}}{b}\right\rceil\right\}.
\end{align}
Notice we do not include the wait to the first bin, $\left\lceil \frac{\ell_1}{b}\right\rceil - 0$, as a bin gap. We could include this if we wish; one additional bin gap will not affect the limiting behavior. We study the gaps between bins, and not between individual summands, because each bin contains at most one summand, and it is natural to view each bin as either `on' or `off'. At the cost of more involved formulas we could deduce similar results about gaps between summands.

In the theorem below we consider all the bin gaps in $(s,b)$-Generacci  legal decompositions of all $m \in [a_{(n-1)b+1},a_{nb+1})$. We let $P_n(g)$ be the fraction of all these bin gaps that are of length $g$ (i.e., the probability of a bin gap of length $g$ among $(s,b)$-Generacci legal decompositions of $m\in [a_{(n-1)b+1},a_{nb+1})$). For example, when considering the $(4,9)$-Generacci sequence notice $m = a_{3} + a_{53} + a_{99} + a_{171} + a_{279}$ with $a_3\in \mathcal{B}_1$, $a_{53}\in \mathcal{B}_6$, $a_{99}\in \mathcal{B}_{11}$, $a_{171}\in \mathcal{B}_{19}$ and $a_{279}\in \mathcal{B}_{31}$, contributes two bin gaps of length 5, one bin gap of length 8, and one bin gap of length 12.

\begin{theorem}[Average Bin Gap Measure for the $(s,b)$-Generacci Sequences]\label{thm:gapstheorem}
For $P_n(g)$ as above, the limit $P(g)$ $:=$ $\lim_{n\to\infty}P_n(g)$ exists. For $g< (s+1)$, $P(g) \ = \ 0$,
and
\begin{align} P(g)\ = \
b (\lambda_{1}^b)^{-g} \ \ \ (g \ge s+1), \end{align}
where $\lambda_1$ is the largest root of $x^{(s+1)b} - x^{sb} - b  = 0$.
\end{theorem}

The proof of Theorem \ref{thm:gapstheorem} is given in \S\ref{sec:gaps}.

We obtain similar results for the individual spacing gap bin measure. We can use the result from \cite{DFFHMPP1} by showing certain combinatorial conditions are met. We quickly review the needed notation from that paper, then state the result.

Given a sequence $\{b_n\}$ and a decomposition rule that leads to unique decomposition, fix constants $c_1,d_1, c_2,d_2$ such that $I_n:= [b_{c_1n+d_1},b_{c_2n+d_2})$ is a well-defined interval for all $n>0$. Below $\delta(x-a)$ denotes the Dirac delta functional, assigning a mass of 1 to $x=a$ and 0 otherwise.

\begin{itemize}

\item \emph{Spacing gap measure:} The spacing gap measure of a $z \in I_n$ with $k(z)$ summands is
\begin{equation}
\nu_{z,n}(x) \ := \ \frac{1}{k(z)-1}\sum_{j=2}^{k(z)}\delta(x-(\ell_j-\ell_{j-1})).
\end{equation}

\item \emph{Average spacing gap measure:} The total number of gaps for all $z\in I_n$ is
\begin{equation}
N_{\rm gaps}(n) \ :=\ \sum_{z=b_{c_1n+d_1}}^{b_{c_2n+d_2}-1}(k(z)-1).
\end{equation}

The average spacing gap measure for all $z\in I_n$ is
\begin{eqnarray} \nu_n(x) & \ :=\ & \frac1{N_{{\rm gaps}}(n)} \sum_{z=b_{c_1n+d_1}}^{b_{c_2n+d_2}-1} \sum_{j=2}^{k(z)} \delta\left(x - (\ell_j - \ell_{j-1})\right) \nonumber\\ & \  = \ & \frac1{N_{{\rm gaps}}(n)} \sum_{z=b_{c_1n+d_1}}^{b_{c_2n+d_2}-1} \left(k(z)-1\right) \nu_{z,n}(x).
\end{eqnarray}
Letting $P_n(g)$ denote the probability of a gap of length $g$ among all gaps from the decompositions of all $m\in I_n$, we have
\begin{equation}
 \nu_n(x) \ = \ \sum_{g=0}^{c_2n+d_2-1} P_n(g) \delta(x - g).
\end{equation}

\item \emph{Limiting average spacing gap measure, limiting gap probabilities:} If the limits exist,  let
\begin{equation} \label{gaplim}
\nu(x) \ = \ \lim_{n\to\infty} \nu_n(x), \ \ \ \ P(g) \ = \ \lim_{n \to \infty} P_n(g).
\end{equation}



\end{itemize}

Although this notation was originally defined for gaps between summands, by taking the $\ell_i$ to represent the gaps between bins, this notation is applicable to our sequences.

\begin{theorem}[Spacing Bin Gap Measure for $(s,b)$-Generacci sequences]\label{indBinGap}
Let $\{a_n\}$ denote the $(s,b)$-Generacci sequence, then for $z\in I_n:=[a_{b(n-1)+1},a_{bn+1})$, the spacing bin gap measures $\nu_{z,n}(x)$ converge almost surely in distribution to the limiting bin gap measure $\nu(x)$.
\end{theorem}

As $\nu(x)=P(x)$, the spacing bin gap measure converges in distribution to  geometric decay behavior.



The same ideas which gave us Gaussian behavior for the Fibonacci Quilt Greedy-6 decomposition from the Gaussian behavior for the $(4,1)$-Generacci sequence also, with trivial tweaking, yield similar results on the average and spacing gap measures. We consider all $m\in I_n:=[q_n,q_{n+1})$,  i.e., those $m$
with a Greedy-6 decomposition beginning with $q_n$.
We let $P_n(g)$ be the fraction of all  gaps from all $m\in I_n$ that are of length $g$.

\begin{theorem}[Average and Spacing Gap Measures for the Greedy-6 Decomposition]\label{thm:FQgapstheorem} Let $\{q_n\}$ denote the Fibonacci Quilt sequence, $P_n(g)$ as above, and consider $m\in I_n:=[q_n,q_{n+1})$. The limit $P(g)$ $:=$ $\lim_{n\to\infty}P_n(g)$ exists and agrees with the $(4,1)$-Generacci limit, and the spacing  gap measures $\nu_{z,n}(x)$ from the Greedy-6 decomposition converge almost surely in distribution to the limiting gap measure from  the (4,1)-Generacci sequence.
\end{theorem}

\subsubsection{New behavior for Fibonacci quilt sequence: $k_{\min}$ vs $k_{\max}$}\label{sec:kmin_kmax}



We do not have unique decompositions with the Fibonacci Quilt sequence.   By Theorem \ref{thm:Dm}, we know that the Greedy-6 algorithm results in a legal decomposition with a minimal number of summands.  Here we investigate the range of the number of summands in any FQ-legal decomposition. 

\begin{definition}
We define $\kmin{m}$ (resp. $\kmax{m}$) to be the smallest (resp. largest) number of summands in any FQ-legal decomposition of $m$.
\end{definition}

The following result gives a lower bound for the growth of $\kmax{m} - \kmin{m}$ which holds for almost all $m \in [q_n, q_{n+1})$ as $n\to\infty$. In particular, we almost always have $\kmax{m} \neq \kmin{m}$. The proof is given in \S\ref{sec:range}.

\begin{thm}\label{thm:kmaxkmin} There is a $C_{\rm FQ} > 0$ such that, as $n\to\infty$, we have $\kmax{m} - \kmin{m} \ge C_{\rm FQ} \log(n)$ for almost all $m\in[q_n,q_{n+1})$. \end{thm}
\section{Gaussian Behavior of Number of Summands}\label{sec:gaussian}

The following sections provide the pieces needed to prove Theorems \ref{thm:gaussian} and \ref{thm:gaussianFQ}. We introduce a new method that allows us to bypass many of the technical obstructions that arise when using standard techniques to handle the determination of the mean and variance in the number of summands. Using this approach we not only can reprove existing results, but also handle new cases such as the $(s,b)$-Generacci and the $\fq$ sequences of this paper. 



\subsection{Proof of Positivity of Linear Terms}\label{positive}

The idea of this section is to reprove and generalize many of the technical results from \cite{MW1} without doing the involved analysis that is needed in order to derive properties of roots of certain polynomials in several variables. In many other papers the methods from \cite{MW1} can be used without too much trouble, as there are explicit formulas available for all the polynomials which arise; however, there are many situations where this is not the case. These difficulties greatly lengthened that paper (and restricted the reach of other works) and resulted in several technical appendices on the behavior of the roots. We avoid these calculations by adopting a more combinatorial view. 

Letting $\{a_n\}$ be any sequence of interest, we prove that the mean and the variance in the number of summands of $m \in [a_n, a_{n+1})$ diverge linearly with $n$. Standard generating function arguments show that the first grows like $Cn + d + o(1)$ and the second like $C'n + d' + o(1)$, where the constants are values of roots of certain associated polynomials (and their derivatives). 
The difficulty in the subsequent analysis of the Gaussianity of the number of summands is that $C$ or $C'$ could vanish. Briefly, the idea behind our combinatorial approach below is that if $C$ 
were to vanish, we would count incorrectly and not have the right number of decompositions. The proof for $C$ (the mean) is very straightforward; the proof for $C'$ (the variance) is more involved, though it essentially reduces to a good approach to counting and then careful book-keeping.

In the arguments below we use $a_n$ to denote the $n$\textsuperscript{th} term of the sequence; we use this and not $G_n$ to emphasize the generality of the results (i.e., the results below are true for more than just PLRS).

\subsection{The Mean}

We introduce some terminology to help us prove results in great generality. Given a length $L$, a {\em segment} of  summands in a generalized Zeckendorf decomposition starting at index $i$ are the summands taken from $\{a_i, a_{i+1}, \dots, a_{i+L-1}\}$; note that for some decomposition rules we may choose a summand with multiplicity. If we write the expansion for $m \in [a_n, a_{n+1})$ we get \be m\ = \ a_{r_1} + a_{r_2} + \cdots + a_{r_{k(m)}}\ee with $a_{r_1} \geq a_{r_2} \geq \cdots \geq a_{r_{k(m)}}$,  where frequently $r_1 = n$. We denote the number of summands of $m$ as $k(m)$, while the number of summands in the segment of length $L$ starting at $i$ is just the number of indices $r_j$ with $i \le r_j < i+L$.

\begin{defi} We say the legal decomposition \emph{acts over a fixed distance} if there is some finite number $f$ such that two segments of a legal decomposition do not interact if they are separated by at least $f$ consecutive summands that are not chosen. This means that whatever summands we have (or do not have) in one segment does not affect our choices in the other, and for the entire decomposition to be legal each of these two segments must be legal. \end{defi}

Note that the sequences we study in this paper both act over a fixed distance. For the $(s,b)$-Generacci sequence we can take $f=sb+1$ and for the Fibonacci Quilt sequence we can take $f=5$. It is also the case that Positive Linear Recurrence relations, which come with a notion of a legal decomposition, act over a fixed distance (we can take $f$ to be at least the length of the recurrence).






The next theorem states that for many generalized Zeckendorf decompositions, $\mu_n$, the average number of summands of integers in $[a_n, a_{n+1})$, is a linear function in $n$ with positive slope, up to an $o(1)$ term which vanishes in the limit.

\begin{thm}\label{thm:gen_mean} Consider an  increasing sequence $\{a_n\}$ which gives rise to unique legal decompositions of the positive integers such that  

\begin{itemize}
\item the rule for the legal decomposition acts over a fixed distance,
\item the average number of summands used for $m \in [a_n, a_{n+1})$ is $\mu_n = Cn + d + o(1)$, and
\item given any constant $A>0$ there is a length $L$ and a probability $p = p(A,L) > 0$  that is less than or equal to the proportion  of legal ways to choose summands in any segment of length $L$ that  have at least $A$ summands, regardless of the choices of summands outside the segment.
\end{itemize}
Then $C > 0$.
\end{thm}

\begin{remark}
Both $(s,b)$-Generacci Sequences and PLRS sequences satisfy all three conditions. To see that $(s,b)$-Generacci Sequences satisfy the third condition, given $A$ if we take $L\geq Asb$ then there is at least one legal way to choose $A$ summands from a segment of length $L$. Hence $p(A,L) > 0$.
\end{remark}

\begin{proof}[Proof of Theorem \ref{thm:gen_mean}]
Assume the claim is false and hence $C = 0$. We show that at least half of the integers have decompositions with at least twice the average number of summands, which contradicts the average number of summands.

For all $n$ sufficiently large, as $C=0$ we have $\mu_n \le 2d$. We choose $A$ to be much larger than $2d$, say $A = 1000(2d+1)$. Let $L$ be large relative to the fixed distance of the decomposition rule (for example, 100 times). For simplicity we assume $n$ is a multiple of $L$ so we may split decompositions up into $n/L$ segments of length $L$, though of course this is not essential and we could just ignore the last segment. We also assume $L$ is large enough so that the third condition holds, namely there is a constant $p(A,L) > 0$ such that in any segment of length $L$ the probability we choose fewer than $A$ summands is at most $1 - p(A,L) < 1$.

We claim that as $n\to\infty$, with probability 1 a decomposition has at least $A$ summands. To see this, we can bound the probability that it has fewer summands by noting that if that were true, it must have fewer than $A$ summands in each of the $n/L$ segments of length $L$. 
Thus
\be {\rm Prob}(m \in [a_n, a_{n+1})\ {\rm has\ less\ than\ }A\ {\rm summands})) \ \le \ \left(1 - p(A,L)\right)^{n/L}.\ee
Thus the probability an $m\in [a_n, a_{n+1})$ has at least $A$ summands tends to 1 as desired:
\be {\rm Prob}(m \in [a_n, a_{n+1})\ {\rm has\ at\ least\ }A\ {\rm summands})) \ \ge \ 1 -\left(1 - p(A,L)\right)^{n/L}.\ee

As $p(A,L) > 0$ is independent of $n$, by taking $n$ sufficiently large at least half of the $m$ in the interval have at least $A$ summands. If we assume all of these have exactly $A$ summands and the rest have 0 then we see that the average number of summands is at least $A/2$, or $500(2d+1)$. As this is far greater than $2d$ we have a contradiction.
\end{proof}

\subsection{The Variance}

We first define additional terminology (especially another notion of legal decompositions) that will help us state our result in great generality.

\begin{definition}
	
	A \textbf{block} is a nonempty finite sequence of nonnegative integers. The \textbf{size} of a block is the sum of the integers in the sequence, while the \textbf{length} of a  block is the number of integers in the sequence.
	
A \textbf{block-batch}, $\mathcal{S}$, is a finite set of blocks with the following characteristics:
\begin{enumerate}
\item[(i)] If two blocks have the same size, then they have the same length,
\item[(ii)]  $\BS$ contains a block of size 0, whose length is minimal among all blocks in $\BS$, and
\item[(iii)]  $\BS$ contains at least one block of size 1.
\end{enumerate}
Property (i) allows us to define a \textbf{length function}: $l(t)$ is the length of all blocks with size $t$.
\end{definition}

\begin{definition} \textbf{(Definition of  $(\mathcal{S}, \mathcal{T})$-legal decompositions)} Consider a strictly increasing sequence of positive integers $\{a_j\}_{j=1}^\infty$.
	Let $\BS$ be a given block-batch  and $\T$ be a given  finite set of blocks. Let $\LT$ be the maximum length of all blocks in $\T$ ($\LT = 0$ if $\T$ is empty). A decomposition of a positive integer $\omega\in\Z$, $\omega=\sum_{i=1}^m c_i a_{m+1-i}$, is {\bf$(\mathcal{S}, \mathcal{T})$-legal} if the coefficient sequence $\{c_i\}_{i=1}^m$ has $c_1>0$,
 the other $c_i \geq  0$, and one of the following two conditions holds:
	\begin{itemize}
		
		\item Condition 1: We have $m\leq  \LT$ and the sequence $\{c_i\}_{i = 1}^m$ is a block in $\T$.
		
		\item Condition 2: There exists $s\geq  1$ such that
		the sequence $\{c_i\}_{i = 1}^s$ is in block-batch $\BS$  and $\{b_i\}_{i=1}^{m-s}$ (with $b_i = c_{s+i}$) is $(\mathcal{S}, \mathcal{T})$-legal or empty.
		
	\end{itemize}
\end{definition}

We observe the following key properties.

\begin{enumerate}

\item If a $(\BS,\T)$-legal decomposition contains a $\T$ type block, then it must be the last block. So any $(\BS,\T)$-legal decomposition contains at most one $\T$ type block.

\item An $(\BS,\T)$-legal decomposition will stay $(\BS,\T)$-legal if an $\BS$ type block is added or removed and indices are shifted accordingly. Only whole blocks can be added and removed. Moreover added blocks cannot be inserted in the middle of existing blocks.

\end{enumerate}

\begin{remark} The usual legal decomposition rules for $(s,b)$-Generacci Sequences and Positive Linear Recurrence Sequences   can be viewed as $(\mathcal{S}, \mathcal{T})$-legal decompositions.     See Appendix D for examples showing how decompositions using several well-known sequences can be viewed as  $(\mathcal{S}, \mathcal{T})$-legal decompositions.
\end{remark}

Let $\Omega_n$ be the set of all $(\BS,\T)$-legal decompositions of integers in $[a_n,a_{n+1})$. Take an $(\BS,\T)$-legal decomposition $\omega\in\Omega_n$ and define the number of summands in the decomposition: $ \nsum({\omega}) = \sum_{i=1}^mc_i$. 
We will define several other random variables that will assist in our study of $\nsum$.
When $n > \LS+\LT$ (with $\LS$ the length of the longest block in $\BS$), there are at least two $\BS$ type blocks in each decomposition. We define the random variable $Z_n$ by setting $Z_n(\omega)$ equal to the size of the  last $\BS$ type block of $\omega \in \Omega_n$. Similarly, we define the random variable $L_n$ by setting $L_n(\omega)$ equal to the length of the  last $\BS$ type block of $\omega \in \Omega_n$.

\begin{theorem}\label{thm:C>0}\label{thm:gen_var}
	Consider a strictly increasing sequence of positive integers $\{a_n\}_{i=1}^\infty$ with $a_{i+1}-a_i\geq a_{j+1}-a_j$ for all $i\geq j$ and $a_{i+1}-a_i>1$ for all $i>\LT+1$, block-batch $\BS$, and set of blocks $\T$ such that all positive integers have unique $(\mathcal{S}, \mathcal{T})$-legal decompositions. If $\E{\nsum} = Cn+d+f(n)$ with $C>0$ and $f(n) = o(1)$, and if $\V{\nsum} = C'n+d'+o(1)$, then
 we can explicitly find $\kappa > 0$, such that $\V{\nsum} \geq  \kappa n$ for all $n \geq  \LT+2$. In other words, $C'> 0$.
\end{theorem}

\emph{We assume the hypotheses of this theorem hold in all lemmas and corollaries below.} Note that $(s,b)$-Generacci  and PLRS Sequences satisfy these hypotheses.

We need additional notation. Let $\ZS$ be the maximum size of all blocks in $\BS$. For all $0 \leq  t \leq \ZS$, define $\B_t$ to be the subset of blocks in $\BS$ whose size is t. For $\mathfrak{b}\in \B_t$, we define $\Upsilon_{n,\mathfrak{b}} = \{\omega\in \Omega_n \mid \text{the last $\BS$ type block is }\mathfrak{b}\}$,


\begin{lemma}\label{lem:collapsingblock}
	Let $n > \LS+\LT$. Define $\phi_{t,\mathfrak{b}}(\omega)$ to be the decomposition that results from removing the  last $\BS$ type block of $\omega$ and shifting indices appropriately.  Then $\phi_{t,\mathfrak{b}}$ is a bijection between $\Upsilon_{n,\mathfrak{b}}$ and $\Omega_{n-l(t)}$.
\end{lemma}

The proof follows by straightforward counting; see Appendix D.

\begin{corollary}\label{cor:var1}
If $\E{\nsum} = Cn+d+f(n)$  then
 \begin{align}\E{\nsum|Z_n = t}\ =\ C(n-l(t)) + d + f(n-l(t)) + t,\label{1}\end{align}
 \begin{align}\E{\nsum^2|Z_n = t}\ =\ \E{Y_{n-l(t)}^2} + 2t[C(n-l(t)) + d + f(n-l(t))] + t^2\label{2},\end{align}
and  when $K_n:=Z_n(\omega)+f(n-L_n(\omega))-CL_n(\omega)$ the \begin{align}\E{K_n}\ = \ f(n)\label{4}\end{align}
\end{corollary}

The proof relies upon the bijection between $\Upsilon_{n,\mathfrak{b}}$ and $\Omega_{n-l(t)}$ which allows us to conclude \linebreak $\E{\nsum|\text{the  last $\BS$ type block is } \mathfrak{b}} = \E{Y_{n-l(t)} + t}$. The final form of the equations are a result of straightforward algebraic manipulation and rules of probability. The complete proof can be found in  Appendix D.

\begin{lemma}
Assume that all integers in $\Omega_n$ have unique $(\BS,\T)$-legal decompositions with respect to the sequence $\{a_n\}$. Then for  $n > \LS+\LT$
	\begin{equation}
		\PP{Z_n = t} \ =\ |\B_t|\frac{H_{n - l(t) + 1} - H_{n - l(t)}}{H_{n+1} - H_n}.
	\end{equation}	
\end{lemma}

\begin{proof} We have
	\begin{equation}\PP{Z_n = t} \ =	\ \sum_{\mathfrak{b}\in\B_t} \frac{|\Upsilon_{n,\mathfrak{b}}|}{|\Omega_n|} \ =\ \sum_{\mathfrak{b}\in\B_t} \frac{|\Omega_{n-l(t)}|}{|\Omega_n|} \ =\  |\B_t| \frac{H_{n - l(t) + 1} - H_{n - l(t)}}{H_{n+1} - H_n}.
	\end{equation}
\end{proof}

\begin{corollary}\label{cor:p0}
Consider a strictly increasing sequence of positive integers $\{a_n\}$ with $a_{i+1}-a_i\geq a_{j+1}-a_j$ for all $i\geq j$. Then for $n > \LS+\LT$, $\PP{Z_n = 0} \geq  1/|\BS|$.
\end{corollary}

The proof is a straightforward application of the lemma; see  Appendix D.

Finally we consider the variance by first using
 $\E{K_n}$ to estimate $\V{K_n}$.

\begin{lemma}\label{lem:varyn}
	For large $n$, $ \V{K_n} > \frac{C^2l(0)^2}{2|\BS|}>0$.
\end{lemma}
\begin{proof}
	For all $n > \LS+\LT$, we have
	\begin{align}
		\V{K_n} &\ =\ \E{K_n^2} - \left(\E{K_n}\right)^2\nonumber\\
		&\ =\ \left(\E{(Z_n - CL_n + f(n - L_n))^2}\right) - \left(f(n)\right)^2\nonumber \\
		&\ =\ \left(\E{(Z_n - CL_n)^2}+\E{2(Z_n - CL_n)\cdot f(n - L_n)} + \E{f(n - L_n)^2}\right) - \left(f(n)\right)^2.
	\end{align}
	Note $0\leq L_n \leq  \LS$ and that $Z_n - aL_n$ is bounded since $ -C\LS \leq  Z_n - CL_n \leq  \ZS$. Also we know $f(n) = o(1)$. Thus
	\[ \lim\limits_{n \to \infty}\E{2(Z_n - CL_n)\cdot f(n - L_n)}\ ,\ \lim\limits_{n \to \infty}\E{f(n - L_n)^2}\ ,\ \lim\limits_{n \to \infty}  \left(f(n)\right)^2 \ =\ 0.\]
Hence
	\begin{equation}\label{5}
		\lim\limits_{n \to \infty}\left(\V{K_n} - \E{(Z_n - CL_n)^2}\right)\ =\ 0.
	\end{equation}
	On the other hand, for all $n > \LS+\LT$ we have
	\begin{align}
		\E{(Z_n - CL_n)^2} &\ =\ \sum\limits_{t = 0}^{\ZS} \PP{Z_n = t} \cdot \left(t - Cl(t)\right)^2\nonumber\\
		&\ \geq \ \PP{Z_n = 0} \cdot \left(0 - Cl(0)\right)^2 \ \geq\  \frac{C^2l(0)^2}{|\BS|},\label{6}
\end{align}
	where the last inequality comes from Corollary \ref{cor:p0}.
	
	By Equation $\eqref{5}$, we know there must exist an $N > \LS+\LT$ such that for all $n > N$, $|\V{K_n} -  \E{(Z_n - CL_n)^2}| < \frac{C^2l(0)^2}{2|\BS|}$, so $\V{K_n} -  \E{(Z_n - CL_n)^2} > -\frac{C^2l(0)^2}{2|\BS|}$. Then $\eqref{6}$ implies $\V{K_n} > \frac{C^2l(0)^2}{2|\BS|}>0$ for all $n > N> \LS+\LT$.
\end{proof}

Finally, we choose $\kappa$. For $N$ as found in Lemma \ref{lem:varyn}, define $\hat{N}:=\max\{\LS+\LT+2,N\}$. Next   let
\begin{equation}
\kappa \ =\ \text{min}\left\{\frac{\V{Y_{\LT+2}}}{{\LT+2}}, \frac{\V{Y_{\LT+3}}}{\LT+2}, \dots, \frac{\V{Y_{\hat{N}}}}{\hat{N}}, \frac{C^2l(0)^2}{2|\BS|\LS}\right\}.\end{equation}
For all $n > \LT+1$, $a_{n + 1} - a_n > 1$, so there are at least two integers in $[a_n, a_{n+1})$. Since the $(\BS,\T)$-legal decomposition of $a_n$ has only one summand while that of $a_n + 1$ has two or more summands, $\V{\nsum}$ is nonzero when $n > \LT+1$. Hence, $\kappa > 0$.

Now we are ready to prove Theorem \ref{thm:C>0}.

\begin{proof}[Proof of Theorem \ref{thm:C>0}] We proceed by strong induction.

{\bf Basis step:} For  $n = \LT+2, \LT+3, \dots, \hat{N}$, $\V{\nsum}>\kappa n$ by definition of $\kappa$.

{\bf Induction step:} Assume $\V{Y_r} \geq  \kappa r$ for $\LT+2 \leq  r < n$. We only need to consider the cases when $n > \hat{N}\geq \LS+\LT+2$.    	So for all $0 \leq t \leq \ZS$, $n > n-l(t) \geq  n - \LS \geq  \LT+2$.

	By \eqref{2}  we have
	\begin{align}
		\E{\nsum^2} &\ =\ \sum\limits_{t=0}^{\ZS}\PP{Z_n = t}\cdot\E{\nsum^2|Z_n = t}\nonumber\\
		&\ =\ \sum\limits_{t=0}^{\ZS}\PP{Z_n = t} \cdot \left(\E{Y_{n-l(t)}^2} + 2t[C(n-l(t)) + d + f(n-l(t))] + t^2\right),\label{ekn2}
	\end{align} and from the inductive hypothesis we have
	\begin{align}
		\E{Y_{n-l(t)}^2} &\ =\ \V{Y_{n-l(t)}} + \left(\E{Y_{n-l(t)}}\right)^2\nonumber\\
	&\	\geq \ \kappa(n-l(t)) + \left(C(n-l(t))+d+f(n-l(t))\right)^2.\label{ind}
\end{align}

Combining \eqref{ekn2} and \eqref{ind} results in an equation with two parts. One is independent on $t$, while the other is of the form of $Z_n + f(n - L_n) - CL_n$, which is exactly $K_n$. We find
	\begin{align}
		\E{\nsum^2} &\ \geq\  \sum\limits_{t=0}^{\ZS}\PP{Z_n = t}  \Bigg[\kappa(n-l(t)) + \left(C(n-l(t))+d+f(n-l(t))\right)^2\nonumber\\
		&\qquad\qquad\qquad\qquad+ 2t[C(n-l(t)) + d + f(n-l(t))] + t^2\Bigg]\nonumber\\
		& \ =\ (Cn+d)^2 + \kappa n + \sum\limits_{t = 0}^{\ZS}\PP{Z_n = t}\cdot\left(t+f(n-l(t))-Cl(t)\right)^2\nonumber\\
		&\qquad\qquad+ 2(Cn+d)\sum\limits_{t=0}^{\ZS}\PP{Z_n = t}\cdot\left(t+f(n-l(t)) -Cl(t)\right) - \kappa\sum\limits_{t=0}^{\ZS}\PP{Z_n = t}\cdot l(t)\nonumber\\
		&\ =\ (Cn+d)^2 + \kappa n + \E{(Z_n+f(n-L_n) -CL_n)^2} + 2(Cn+d)f(n) - \kappa\E{L_n},\label{ekn3}
	\end{align} with the last equality coming from $\eqref{4}$.
	
Finally, \eqref{ekn3}, the definition of $\kappa$, and  Lemma \ref{lem:varyn} imply
	\begin{align}
		\V{\nsum} - \kappa n &\ =\ \E{\nsum^2} - \left(\E{\nsum}\right)^2 - \kappa n\nonumber\\
		&  \ \geq\  \E{(Z_n+f(n-L_n) -CL_n)^2} - \kappa\E{L_n}-(f(n))^2\nonumber\\
		&\ =\ \E{K_n^2} - \kappa\E{L_n} - \left(\E{K_n}\right)^2\nonumber\\
		&\ =\ \V{K_n} - \kappa\E{L_n}\nonumber\\
		& \ \geq\  \V{K_n} - \kappa\LS\nonumber\\
		& \ \geq\ \frac{C^2l(0)^2}{2|\BS|}-\frac{C^2l(0)^2}{2|\BS|\LS}\LS \ =\ 0,
	\end{align} and therefore $\V{\nsum} \geq  \kappa n$.
\end{proof}


\subsection{Generating Function for $(s,b)$-Generacci Legal Decompositions}\label{sec:gaussianSB}
Let $p_{n,k}$ (with $n,k\geq0$) denote the number of $m\in[a_{(n-1)b+1},a_{nb+1})$ whose $(s,b)$-Generacci legal decomposition contains exactly $k$ summands, where $a_{nb+1}$ is the first entry in the $(n+1)^{\text{st}}$ bin of size $b$.

\begin{proposition}\label{recurrencefromhell}
Let $n,k\geq 0$. Then
\begin{align}
p_{n,k}&\ = \ \begin{cases}
1&\mbox{{\rm if} $n=k=0$}\\
b&\mbox{{\rm if} $1\leq n\leq s$ and $k=1$}\\
b\cdot q_{n-(s+1),k-1}&\mbox{{\rm if} $n\geq s+1$ {\rm and} $1\leq k\leq \frac{n+s}{s+1}$}\\
0&\mbox{{\rm otherwise,}}
\end{cases}
\label{FQpnk}\end{align}
where $q_{n,k} $ (with ${n,k} \geq 0$) is the number of $m \in [0,a_{nb+1})$ whose $(s,b)$-Generacci legal decomposition contains exactly $k$ summands. Set $F(x,y)=\sum_{n=0}^{\infty} \sum_{k=0}^{n_{*}}p_{n,k}x^ny^k$ with $n_*=\lceil\frac{n+s}{s+1}\rceil$. Then
\begin{align}F(x,y)&\ = \ 1 + \frac{byx}{1-x-byx^{s+1}}.\label{functionF(x,y)}\end{align}
\end{proposition}

We omit the proof here as the details follow from standard bookkeeping and algebraic manipulation.  The proof of this proposition 
is found in Appendix A.


To complete the proof of Theorem \ref{thm:gaussian} we make use the following result from \cite{DDKMV}.


\begin{theorem} \label{thm:DDKMV}\cite[Theorem 1.8]{DDKMV} Let $\kappa$ be a fixed positive integer. For each $n$, let a discrete random variable $Y_n$ in $I_{n}=\{1,2,\ldots,n\}$  have
\begin{align}
{\rm Prob}(Y_n=j)\ = \ \begin{cases}p_{j,n}/\sum_{j=1}^np_{j,n}&\text{{\rm if} $j\in I_n$}\\ 0&\text{{\rm otherwise}}\end{cases}\label{prob}
\end{align}
for some positive real numbers $p_{1,n}, p_{2,n}, \ldots, p_{n,n}$. Let $g_n(y):=\sum_j p_{j,n}y^j$.

If $g_n$ has the form $g_n(y) = \sum_{i=1}^\kappa q_i(y)\alpha_i^n(y)$ where
\begin{enumerate}
\item[(i)]  for each $i \in \{1, \ldots, \kappa\}, q_i, \alpha_i: \mathbb{R} \to \mathbb{R}$ are three times differentiable functions which do not depend on $n$;
\item[(ii)]  there exists some small positive $\epsilon$ and some positive constant $\lambda < 1$ such that for all $y \in I_{\epsilon} = [1-\epsilon, 1 + \epsilon], |\alpha_1(y)| > 1$ and $|\frac{\alpha_i(y)}{\alpha_1(y)}| < \lambda < 1$ for all $i=2, \ldots, \kappa$;
\end{enumerate}
then
\begin{enumerate}
\item the mean $\mu_n$ and variance $\sigma_n^2$ of $Y_n$ both grow linearly with $n$. Specifically,
\begin{align}
\label{DDKMV-form}
\mu_n \ =\ Cn + d + o(1), \ \ \ \ \sigma_n^2 \ =\ C^\prime n + d^\prime + o(1)
\end{align}
where

\begin{align}
C&\ = \  \frac{\alpha_1'(1)}{\alpha_1(1)}, \  d \ = \  \frac{q_1'(1)}{q_1(1)} \nonumber\\
C^\prime &\ = \  \frac{d}{dy}\left. \left(\frac{y\alpha_1'(y)}{\alpha_1(y)} \right)\right\vert_{y=1} \ = \  \frac{\alpha_1(1)[\alpha_1'(1)+ \alpha_1''(1)]-\alpha_1'(1)^2}{\alpha_1(1)^2}\nonumber\\
d^\prime &\ = \  \frac{d}{dy}  \left. \left(\frac{yq_1'(y)}{q_1(y)} \right) \right\vert_{y=1} \ = \  \frac{q_1(1)[q_1'(1)+ q_1''(1)]-q_1'(1)^2}{q_1(1)^2}.
\end{align}
\end{enumerate}
Moreover, if
\begin{enumerate}
\item[(iii)]    $\alpha_1'(1) \neq 0$ and $\frac{d}{dy}\left[ \frac{y\alpha_1'(y)}{\alpha_1(y)}\right]|_{y=1} \neq 0$, i.e., $C,C'>0$,
\end{enumerate}
then
\begin{enumerate}
\item[(2)] as $n \to \infty$, $Y_n$ converges in distribution to a normal distribution.
\end{enumerate}
\end{theorem}

To apply Theorem \ref{thm:DDKMV} we still need some auxiliary results about the function $g_n(y)$ which gives the coefficient of $x^n$ in the expansion of the generating function $F(x,y)$. In fact we need results regarding the partial fraction decomposition of $1/(1-x-byx^{s+1})$.

\begin{lemma}\label{NoRepeats2}Let $s,b\geq 1$ and $y>0$. Let  $f(x)=1-x-byx^{s+1}$. Then
\begin{enumerate}
\item $f(x)$ has no repeated roots,
\item $f(x)$ has a positive root  $\lambda_1(y)$ whose modulus is smaller than the modulus of any other root of $f(x)$. Moreover, $\lambda_1(y)<1$.
\end{enumerate}
\end{lemma}
\begin{proof} (1) Let $h(x)=x^{s+1}+ax-a$, where $a =1/by$, and suppose that $h(x)$ has a repeated root, say $r$ (note $r \neq 0$). Then
$h(r)$ and $h'(r)$ equal 0 yields a contradiction. (2) To find the roots of $f(x)=1-x-byx^{s+1}$ we use the change of variable $w=1/x$ and note that the roots of
\begin{align}g(w)\ = \ w^{s+1}-w^s-by\end{align}
are the eigenvalues of the companion matrix of the polynomial $g(w)$. This matrix is a non-negative irreducible matrix so by the Perron-Frobenius Theorem, $g(w)$ has a unique positive dominant root $\mu(y)$. Hence $\lambda_1(y):=\frac{1}{\mu(y)}$ is the unique positive root  of $f(x)$ with smallest modulus. Now by applying the Intermediate Value Theorem we note that one of the positive roots lies in the interval $[0,1]$. 
 Since $\lambda_1(y)$ is the smallest positive root, then clearly $0<\lambda_1(y)<1$.
\end{proof}


\begin{proposition}\label{gny}
Let $g_n(y)=\sum_{k=0}^{\infty}p_{n,k}y^k$, which is the coefficient of $x^n$ in the generating function of the $p_{n,k}$'s. Then for sufficiently large $n$
\begin{align}
\label{eq.gny}
g_n(y)\ = \ \displaystyle\sum_{i=1}^{s+1}q_i(y)\alpha_i^n(y),\end{align}
where  for $1\leq i\leq s+1$, $\alpha_i(y)=\frac{1}{\lambda_i(y)}$ with $\lambda_i(y)$ the distinct roots of the polynomial $f(x)=1-x-byx^{s+1}$ and $q_i(y)$ are algebraic functions of $y$ which depend on these roots.
\end{proposition}

\begin{proof}
Let $\lambda_1(y),\lambda_2(y),\ldots,\lambda_{s+1}(y)$ be the distinct roots of $f(x)=1-x-byx^{s+1}$. Using a partial fraction decomposition of $1/f(x)$, 
\begin{align}
\frac{1}{f(x)}& \ = \ \sum_{i=1}^{s+1}\frac{p_i(y)}{x-\lambda_i(y)},
\end{align}
where $p_i(y)$ are algebraic functions of $y$ depending on $\lambda_i(y)$.
By rewriting the terms and using the geometric sum formula we have that
\begin{align}
\frac{1}{f(x)}&\ =\ \sum_{i=1}^{s+1}  \hat{p}_i(y)\frac{1}{1-\frac{x}{\lambda_i(y)}}\ =      \ \sum_{i=1}^{s+1}\sum_{n=0}^{\infty} \hat{p}_{i}(y)\left(\alpha_i(y)x\right)^n \ = \ \displaystyle \sum_{n=0}^{\infty}\left[\sum_{i=1}^{s+1} \hat{p}_i(y)\alpha_i^n(y)\right] x^n,
\end{align}
where $\hat{p}_i(y)=-\frac{p_i(y)}{\lambda_{i}(y)}$ and $\alpha_i(y)=\frac{1}{\lambda_i(y)}$.
So
\begin{align}
F(x,y) \ = \ \frac{1+x(by-1)-byx^{s+1}}{f(x)} \ =  \ \left({1+x(by-1)-byx^{s+1}}\right)\sum_{n=0}^{\infty}\left[\sum_{i=1}^{s+1} \hat{p}_i(y)\alpha_i^n(y)\right] x^n.
\end{align}
Thus for sufficiently large $n$,
\begin{align}
g_n(y)&\  = \  \sum_{i=1}^{s+1} \alpha_i^n(y)\left[ \hat{p}_i +   (by-1) \hat{p}_i \alpha_i^{-1}(y)   -by \hat{p}_i \alpha_i^{-s-1}(y)  \right] \ =\ \sum_{i=1}^{s+1}q_i(y)\alpha_i^n(y).
\end{align}
\end{proof}

\begin{proof}[Proof of Theorem \ref{thm:gaussian}]
To prove Gaussianity we need only show that $g_n(y)$ satisfies conditions (i)--(iii)  in Theorem \ref{thm:DDKMV}.

\begin{itemize}
\item Condition (i): For each $i \in \{1, \ldots, s+1\}$,  $q_i(y)$ and $\alpha_i(y)$ are three times differentiable functions as roots of polynomials are differentiable functions of the polynomial coefficients, see \cite{Lozada}.
\item Condition (ii): Follows from Lemma \ref{NoRepeats2}.
\item Condition (iii): Follows from Theorems \ref{thm:gen_mean} and \ref{thm:gen_var}.
\end{itemize}

Therefore, by satisfying the conditions of Theorem \ref{thm:DDKMV}, we have completed our proof.
\end{proof}

\section{Gap Measures for the $(s,b)$-Generacci Sequences}


\subsection{Average Bin Gap Measure}\label{sec:gaps}
\begin{proof}[Proof of Theorem \ref{thm:gapstheorem}]
Let $m \in I_n:= [a_{(n-1)b+1},a_{nb+1})$ have legal decomposition
\begin{align}
m\ = \ a_{\ell_1} + a_{\ell_2}  + \cdots + a_{\ell_k} \ {\rm with} \ \ell_1> \ell_2 >  \cdots >  \ell_k \  {\rm and} \ a_{\ell_i}\in \mathcal{B}_{\left\lceil\frac{\ell_i}{b}\right\rceil}\ {\rm for\ all\ } 1\leq i\leq k. \end{align}

Recall that $P_n(g)$ is the fraction of bin gaps that are of length $g$ (i.e., the probability of a bin gap of length $g$ among $(s,b)$-Generacci legal decompositions of $m\in [a_{(n-1)b+1},a_{nb+1})$). Clearly $P_n(g)=0$ whenever $g<s+1$ since we must skip $s$ bins between summands. For $g\geq s+1$, define $X_{i,g}$ as the number of $m\in I_n$ whose decompositions contribute a bin gap of length $g$ starting at bin $\mathcal{B}_i$. Then

\begin{align}P_n(g)&\ = \ \displaystyle \frac{ \sum_{i=1}^{n}X_{i,g}}{(\mu_n-1)([a_{nb+1}]-[a_{(n-1)b+1}])}.\end{align}

To compute $X_{i,g}$ note we have a summand from bin $\mathcal{B}_i$ and one from $\mathcal{B}_{i+g}$, and no summands from $\mathcal{B}_{i+1},\mathcal{B}_{i+2},\ldots,\mathcal{B}_{i+g-1}$. Moreover since $m\in I_n=[a_{(n-1)b+1},a_{nb+1})$, $m$ must contain a summand from $\mathcal{B}_n$. Hence there is freedom to choose summands from $\mathcal{B}_1,\mathcal{B}_2,\ldots,\mathcal{B}_{i-s-1}$ and then again we are free to choose summands from bins
 $\mathcal{B}_{i+g+s+1}$, $\mathcal{B}_{i+g+s+2}, \ldots,\mathcal{B}_{n-s-1}.$

The number of ways to choose legally from $\mathcal{B}_1,\mathcal{B}_2,\ldots,\mathcal{B}_{i-s-1}$ is $a_{(i-s-1)b+1}-1$. Similarly,  the number of ways to choose legally from $\mathcal{B}_{i+g+s+1},\mathcal{B}_{i+g+s+2},\ldots,\mathcal{B}_{n-s-1}$ is the number of integers in $[0,  a_{(n-2s-g-i-1)b+1}-1)$. As we selected summands from $\mathcal{B}_i, \mathcal{B}_{i+g}$ and $\mathcal{B}_n$,
\begin{align}X_{i,g}\ = \ b^3[a_{(i-s-1)b+1}-1][a_{(n-2s-g-i-1)b+1}-1].\label{binapprox}\end{align}

By Equation \eqref{explicitConstant_2} of Theorem \ref{thrm:recurrencesb_2},
\begin{align}
a_n
&\ = \ c_1\lambda_1^n(1+O(\varepsilon^n)),\label{an_explicit_vareps}
\end{align}
where $\varepsilon = |\lambda_2/\lambda_1|$,  for some constants $c_1$, $\lambda_1$, and $\lambda_2$, where $\lambda_1>1$, $c_1>0$ and $|\lambda_2|<\lambda_1$.
Thus
\begin{align}
X_{i,g}& =  b^3c_{1}^2\lambda_{1}^{(n-3s-2)b+2}(\lambda_{1}^b)^{-g}(1+O(\varepsilon^{(i-s-1)b+1}))(1+O(\varepsilon^{(n-i-2s-g-1)b+1})).
\end{align}

We break the sum into three ranges: $i \le 8\log n$, $8
\log n < i < n-8\log n$, and $n-8\log n \le i \le n$. Note that for $8 \log n < i < n - 8\log n$,
\begin{align}
\varepsilon^{(i-s-1)b+1},
\varepsilon^{(n-i-2s-g-1)b+1}  \leq   \varepsilon^{4\log n},
\end{align}
which implies that all lower order terms are negligibly small relative to the main term. On the other hand
\begin{align}
\displaystyle \sum_{1 \leq i < 8\log n} X_{i,g}&=b^3c_{1}^2\lambda_{1}^{(n-3s-2)b+2}(\lambda_{1}^b)^{-g}O(\log n)\nonumber\\
\displaystyle\sum_{n- 8\log n \leq i \leq n} X_{i,g} &=b^3c_{1}^2\lambda_{1}^{(n-3s-2)b+2}(\lambda_{1}^b)^{-g}O(\log n).
\end{align} Hence
\begin{align}
P_n(g)&\ = \ \frac{\displaystyle  \sum_{1 \leq i < 8\log n} X_{i,g} \  \  + \sum_{8\log n  \le i < n-8\log n} X_{i,g} \  \ + \sum_{n- 8\log n \leq i \leq n} X_{i,g} }{(\mu_n-1)([a_{nb+1}]-[a_{(n-1)b+1}])}\nonumber\\
&\ = \ \frac{ b^3c_{1}^2\lambda_{1}^{(n-3s-2)b+2}(\lambda_{1}^b)^{-g}\left[O(\log n) + (n-16\log n)\left(1 + O(\varepsilon^{4\log n})\right)\right]}{Cn(c_{1}\lambda_{1}^{nb+1}-c_{1}\lambda_{1}^{(n-1)b+1})}\nonumber\\
&\ = \ \frac{ b^3c_{1}^2\lambda_{1}^{(n-3s-2)b+2}}{Cnc_{1}\lambda_{1}^{(n-1)b+1}(\lambda_{1}^{b}-1)} (\lambda_{1}^b)^{-g} \left[n + O(\log n)\right].\nonumber\\
\end{align}
Taking the limit as $n\to\infty$ yields
\begin{align}P(g) \ = \ \frac{b^3c_{1}}{C(\lambda_{1}^b-1)\lambda_{1}^{(3s+1)b-1}}(\lambda_{1}^b)^{-g}.
\label{valuePg_sb}
\end{align}
As $P(g)$ defines a  probability distribution and  $P(g) = 0$ for  $g < s+1$, $\sum_{g = s+1}^\infty P(g) = 1$.
Evaluating the geometric series and using $\lambda_1$ is a root of $x^{(s+1)b} - x^{sb} - b  = 0$ yields
\begin{align}
\frac{b^3c_{1}}{C(\lambda_{1}^b-1)\lambda_{1}^{(3s+1)b-1}}\  = \ b.
\end{align}
Thus $P(g) = b(\lambda_1^b)^{-g}$.
\end{proof}


\subsection{Spacing Bin Gap Measure}\label{sec:individualgaps}
We prove Theorem~\ref{indBinGap} by checking that the conditions of \cite[Theorem 1.1]{DFFHMPP1} are satisfied by the spacing bin gap measure of the $(s,b)$-Generacci sequence; note we are working with gaps between bins and not summands, but by collapsing a bin we find the arguments are identical. We restate \cite[Theorem 1.1]{DFFHMPP1} below for ease of reference.

\begin{theorem}\label{genthm2}\cite[Theorem 1.1]{DFFHMPP1} For $z\in I_n:=[a_{c_1n+d_1},a_{c_2n+d_2})$, the individual gap measures $\nu_{z,n}(x)$ converge almost surely in distribution  to the average gap measure $\nu(x)$ if the following hold.

\begin{enumerate}
\item The number of summands for decompositions of $z \in I_n$ converges to a Gaussian with mean $\mu_n = \cm n+O(1)$ and variance $\sigma_n^2 = \cv n+O(1)$, for constants $\cm, \cv>0$, and $k(z) \ll n$ for all $z \in I_n$.\label{condition1}

\item We have the following, with $\lim_{n\to\infty} \sum_{g_1, g_2} {\rm error}(n,g_1,g_2) = 0$:
\begin{align}
\frac{2}{|I_n|\mu_n^2}\sum_{j_1<j_2}X_{j_1,j_1+g_1,j_2,j_2+g_2}(n) \ = \  P(g_1)P(g_2)+\text{{\rm error}}(n,g_1,g_2).
\end{align}\label{condition2}


\item The limits in Equation \eqref{gaplim} exist.\label{condition3}
\end{enumerate}
\end{theorem}

In \cite{DFFHMPP1}, the authors used the following definition: for $g_1,g_2 \ge 0$
\begin{align}
X_{j_1,j_1+g_1, j_2,j_2+g_2}(n)  & \ := \  \#\left \{
z \in I_n:\begin{subarray}\ b_{j_1},\ b_{j_1+g_1},\ b_{j_2},\ b_{j_2+g_2}\ \text{in}\ z\text{'s\ decomposition,}\\
\text{but\ not\ } b_{j_1+q},\ b_{j_2+p}\ \text{for}\ 0<q<g_1,\ 0<p<g_2
\end{subarray}
\right \}.
\end{align}
Since we are concerned with the gaps between bins we will compute $X_{j_1,j_1+g_1, j_2,j_2+g_2}(n)$ by counting $z \in I_n$ whose decomposition has  a summand from bins $\mathcal{B}_{j_1}$ and $\B_{j_1+g_1}$ (with no bins used in between) and again from bins $\mathcal{B}_{j_2}$ and $\B_{j_2+g_2}$ (with no bins used in between).

\begin{proposition}\label{prop:error}
We have
\begin{align}
\frac{2}{|I_n|\mu_n^2}\displaystyle\sum_{j_1<j_2}X_{j_1,j_1+g_1, j_2,j_2+g_2}(n)\ =\ P(g_1)P(g_2)+\text{error}(g_1,g_2,n)
\end{align}
where the error as $n\to\infty$ summed over all pairs $(g_1, g_2)$ goes to zero.
\end{proposition}


\begin{proof}
Assume $j_1<j_2$. We compute $X_{j_1,j_1+g_1, j_2,j_2+g_2}(n)$: We  take a summand each from bins $\mathcal{B}_{j_1}$ and $\B_{j_1+g_1}$  and again from bins $\mathcal{B}_{j_2}$ and $\B_{j_2+g_2}$, and finally since $z\in I_n=[a_{(n-1)b+1},a_{nb+1})$,  $z$ must contain a summand from bin $\mathcal{B}_n$. Additionally, we have freedom in selecting summands from bins $\B_1,\B_2,\ldots,\B_{j_1-(s+1)}$,
then from bins  $\B_{j_1+g_1+(s+1)},\B_{j_1+g_1+(s+2)},\ldots,\B_{j_2-(s+1)}$, and lastly from bins $\B_{j_2+g_2+(s+1)},\B_{j_2+g_2+(s+2)},\ldots,\B_{n-(s+1)}.$

The number of ways to choose summands legally from  $\B_1,\B_2,\ldots,\B_{j_1-(s+1)}$ is $a_{(j_1-s-1)b+1}-1$; the number of ways to choose summands legally from $\B_{j_1+g_1+(s+1)}$, $\B_{j_1+g_1+(s+2)}\ldots,\B_{j_2-(s+1)}$ is $a_{(j_2-j_1-g_1-2s-1)b+1}-1$; the number of ways to choose summands legally from  $\B_{j_2+g_2+(s+1)},$ $\B_{j_2+g_2+(s+2)},\ldots,\B_{n-(s+1)}$ is given by   $a_{(n-j_2-g_2-2s-1)b+1}-1$.  Hence

\begin{align}X_{j_1,j_1+g_1,j_2,j_2+g_2}(n)
&\ = \ b^5[a_{(j_1-s-1)b+1}-1][a_{(j_2-j_1-g_1-2s-1)b+1}-1][a_{(n-j_2-g_2-2s-1)b+1}-1].\nonumber\\
\label{gapapprox}
\end{align}
Using the explicit form for the terms of the $(s,b)$-Generacci sequence given in Equation \eqref{an_explicit_vareps},
Equation \eqref{gapapprox} yields
\begin{align}
X_{j_1,j_1+g_1,j_2,j_2+g_2}(n)& =   b^5c_{1}^3\lambda_{1}^{(n-5s-3)b+3}\left(\lambda_{1}^{b}\right)^{-(g_1+g_2)}(1+O(\varepsilon^{j_1b}))(1+O(\varepsilon^{(j_2-j_1)b}))(1+O(\varepsilon^{(n-j_2)b})),\end{align}
where it is important to recall that $\varepsilon<1$.

\begin{align}
\mbox{Let \ } S_n  &= \ \{(j_1,j_2): \ 1 \leq j_1 < j_2 \leq n\},  \mbox{and}\nonumber\\
T_n  &= \ \{(j_1, j_2) \in S_n: \  8 \log n \le j_1 < j_2 < n - 8 \log n, \ j_2 - j_1 > 8\log n\}. \nonumber
\end{align}
Then for $(j_1, j_2) \in T_n$,  \ $\varepsilon^{j_1b} \leq \varepsilon^{8\log n}, \
\varepsilon^{(j_2-j_1)b} \leq \varepsilon^{8\log n}, \
\varepsilon^{(n-j_2)b} \leq \varepsilon^{8\log n},$ which implies that all lower order terms are negligibly small relative to the main term. Also, note  that the  sum of $1$ over all   $(j_1, j_2) \in S_n \setminus T_n$ is of order $n \log n$. Thus
\begin{align}
\displaystyle \frac{2\displaystyle\sum_{(j_1, j_2) \in S_n}X_{j_1,j_1+g_1, j_2,j_2+g_2}(n)}{|I_n|\mu_n^2}
&= \ \displaystyle \frac{2\left(\displaystyle\sum_{(j_1, j_2) \in T_n}\hspace{-.1in}X_{j_1,j_1+g_1, j_2,j_2+g_2}(n)  \ +  \displaystyle\sum_{(j_1, j_2) \in S_n \setminus T_n}\hspace{-.2in}X_{j_1,j_1+g_1, j_2,j_2+g_2}(n)\right)}{|I_n|\mu_n^2}\nonumber\\
 &\hspace{-1.1in}= \ \displaystyle \frac{b^6c_1^2}{C^2\lambda_1^{(6s+2)b-2}(\lambda_1^b-1)^2}(\lambda_1^b)^{-g_1}(\lambda_1^b)^{-g_2} \frac{\lambda_1^{sb} (\lambda_1^b - 1)}{b} \left[1 + O(\log n / n)\right]   \nonumber\\
 &\hspace{-1.1in}    = \ \displaystyle P(g_1) P(g_2)\left[1 + O(\log n / n)\right] ,
\end{align}
the last equality follows immediately from \eqref{valuePg_sb} and the fact that $\lambda_1$ is the largest root of the characteristic equation $x^{(s+1)b} - x^{sb} - b = 0$, the defining relation of our sequence $\{a_n\}$. As $P(g_1)P(g_2)$ sums to 1, the sum of the error term over all pairs $(g_1, g_2)$ goes to zero as required.
\end{proof}

\begin{proof}[Proof of Theorem \ref{indBinGap}] We simply need to  check that Conditions (1)--(3) of Theorem \ref{genthm2} hold. First we note that letting $c_1=b$, $d_1=1-b$ and $c_2=b$ and $d_2=1$, implies that the interval of interest is $I_n=[a_{b(n-1)+1},a_{bn+1})$.
Then Theorem \ref{thm:gaussian} shows the first part Condition \eqref{condition1} is satisfied. Now note that there are $n-1$ allowable bins from which to select summands and any $z\in I_n$ will have at most $\left\lceil\frac{n-1}{s+1}\right\rceil$ summands as there must be $s$ bins between each summand selected.
Hence for any $z\in I_n$, $k(z)\leq \left\lceil\frac{n-1}{s+1}\right\rceil<n$ which completes the proof that Condition \eqref{condition1} is satisfied. 
Condition \eqref{condition3} follows from Theorem \ref{thm:gapstheorem}. Finally, Condition \eqref{condition2}  follows from Proposition \ref{prop:error}.
\end{proof}



\section{Gaussianity and Gap Measures for Fibonacci Quilt}

The $(4,1)$-Generacci sequence yields Gaussian and Gap Measure results for  Greedy-6 decompositions. The Greedy-6 decomposition is almost the same as the legal decomposition from the $(4,1)$-Generacci sequence as the gap between almost all summands in a Greedy-6 decomposition is at least 5. The only difference is that for the Greedy-6 decomposition the last two summands can have indices differing by 2 (if that happens the subsequent index is at least 6 larger). This possible gap of length 2 does not matter in the limit.

\begin{proof}[Proof of Theorem \ref{thm:gaussianFQ}]  As the two decompositions are so similar, the Gaussianity result for the Greedy-6 decomposition  follows from that for the $(4,1)$-Generacci sequence. We  partition our integers $m$ into two distinct sets where the Greedy-6 decomposition $\mathcal{G}(m)$  starts with $q_n$ and either:

\begin{itemize}

\item ends with $q_4+q_2$ and the third smallest summand is at least $q_{10}$; or





\item all indices differ by at least 5.
\end{itemize}
Both of these cases have Gaussian behavior by Theorem \ref{thm:gaussian} specified to the $(4,1)$-Generacci Sequence. 

In the first case, Greedy-6 decompositions must have the summand $q_n$ as well as $q_4$ and $q_2$. Thus we do not have $q_1, q_3, q_5, q_6, q_7, q_8$ or $q_9$, but $q_{10}$ is possible. Define $$Q_{n,\alpha}:=\{m \in [q_n,q_{n+1})\mid  q_2 \  \text{and} \  q_4 \ \text{are summands in the Greedy-6 decomposition of} \ m \}.$$

Consider the $(4,1)$-Generacci sequence $\{a_n\}.$  Define the set of integers $$J_{n,\alpha} := \{\omega \in [a_n, a_{n+1})\mid a_1, a_2, \dots, a_9 \ \  \text{are not in the decomposition of } \omega \}.$$  As the integers in $J_{n,\alpha}$ decompose with $a_n$ as the largest summand and any legal set of summands from  $\{ a_{n-5}, a_{n-6}, \dots, a_{10}\},$ we have $|J_{n,\alpha}| = a_{n-14} - 1.$ Moreover, the bijection between the sets $J_{n,\alpha}$ and  $[1, a_{n-14})$ preserves the number of summands in a decomposition. As the number of summands in the $(4,1)$-Generacci legal decomposition of an integer 
from $[1, a_{n-14})$ is Gaussian, the number of summands in the $(4,1)$-Generacci legal decomposition of an integer 
from $J_{n,\alpha}$ is Gaussian.
There is a bijection between the sets $J_{n,\alpha}$ and $Q_{n,\alpha}$ that exactly increases the number of summands in a decomposition by 2,  hence the number of summands in the Greedy-6 legal decomposition of an integer chosen uniformly at random from $Q_{n,\alpha}$ is Gaussian.  The mean and variance of each of this Gaussian will differ  from the mean and variance of the $(4,1)$-Generacci sequence in the constant term, but as the mean and the variance  are of the form $An + B +o(1)$ and $C n + D +o(1)$, this shift does not matter in the limit.

All $m$ in the second case are in a bijection with all $\omega \in [a_n, a_{n+1})$ that precisely preserves the indices in the decompositions of $m$ and $\omega$. Hence the number of summands of such an $m$ is Gaussian.

Combining these two Gaussians distributions results in an overall Gaussian.
\end{proof}

\begin{proof}[Proof of Theorem \ref{thm:FQgapstheorem}] We first note that the proportion of gaps of length 2 is negligibly small  as $n\to\infty$. The number of gaps of a typical element is strongly concentrated on the order of $n$, so one extra gap of length 2 is proportionally only on the order of $1/n$, and thus in the limit will have zero probability.

For the remaining gap sizes, we break this problem into two cases as we did in the proof of Theorem \ref{thm:gaussianFQ}. We then argue identically as in the $(4,1)$-Generacci case, and note that our proofs were entirely combinatorial; all that mattered was the number of ways to choose summands satisfying the legal rule.
 \end{proof}

\begin{rek} Note the utility of this perspective suggests some natural future questions: as the Fibonacci Quilt's Greedy-6 decomposition is just the $(4,1)$-Generacci with a tweak in the beginning, do other tweaks lead to geometrically interesting sequences?\end{rek}



\section{Range of Number of Summands for Fibonacci Quilt Decompositions}\label{sec:range}

We introduce the notion of gap strings to clean up manipulations by eliminating the need for cluttering the paper with sums and subscripts.

\begin{definition}
Let $q_{\ell_1} + q_{\ell_{2}} + \dots + q_{\ell_t}$ be any  decomposition of $m$ with $q_{\ell_i} \geq q_{\ell_{i+1}}$ for $i = 1, 2, \dots , t-1$.  The {\em gap string} of the decomposition is the $(t-1)$-tuple \begin{align}(\ell_1 - \ell_{2}, \ell_{2} - \ell_{3}, \dots, \ell_{t-1} - \ell_{t}).\end{align}
\end{definition}


From Theorem \ref{thm:Dm} we know the number of summands in the Greedy-6 decomposition of any $m$ is minimal and the corresponding gap string $(x_1, x_2, \dots, x_{\kmin{m}-1})$ has $x_i \geq 5$ for all $i$ except possibly $x_{\kmin{m}-1} = 2$ (i.e., the Greedy-6 decomposition used $q_4 + q_2$).

\begin{proof}[Proof of Theorem \ref{thm:kmaxkmin}]
From Theorem \ref{thm:Dm} the Greedy-6 decomposition is a minimal decomposition (i.e., no other legal Fibonacci Quilt decomposition uses fewer summands). We investigate  how many $m \in [q_n, q_{n+1})$ have $\kmax{m}-\kmin{m} \ge g(n)$ for a fixed function $g(n)$, and then see how large we may take it while ensuring the inequality holds for almost all $m$ in the interval. The argument below was chosen as it gives the optimal growth rate of $g(n)$ but not the optimal constant; with a little more work the value of $C_{\rm FQ}$ could be slightly increased, but a growth rate of essentially $\log(n)$ is the natural boundary of this approach.

Let $\mathcal{G} = (5,5, 10, 5, 5, 10, \dots, 5,5, 10)$ be a fixed gap pattern among $3g(n)+1$ addends. Note that the number of summands in a  decomposition of $m\in I_n$ can be increased by $g(n)$ if  the decomposition has a gap string that contains the substring
$\mathcal{G}$  beginning at  $q_{A + 20g(n)}$ with $10 +  20g(n) \leq A + 20g(n) \leq n$. Using recurrence relations ($q_n+q_{n-2}=q_{n+1}+q_{n-5}$ and $q_n+q_{n-4}=q_{n+1}$ proved in \cite{CFHMN2}) we get a new  FQ-legal decomposition of $m$ where the only difference is that substring $\mathcal{G}$ is replaced with the substring $\mathcal{G}^\prime = (6,2,7,5,6,2,7,5, \dots, 6, 2, 7, 5)$.\footnote{For example, replacing string $(5, 5, 10)$ with $(6, 2, 7, 5)$ can be seen as $q_{30+\ell}\,+ \,q_{25+\ell}\,+\, q_{20+\ell}\,+\, q_{10+\ell}\,=\, q_{30+\ell}\,+\, q_{24+\ell}\,+\, q_{22+\ell}\,+\, q_{15+\ell}\,+\, q_{10+\ell}$.}  The  starting and ending summands remain the same but there are now $4g(n)+1$ summands indicated by the gap substring. Hence for such $m$, $\kmax{m} - \kmin{m} \geq g(n)$.

We break the set of Fibonacci  Quilt summands $\{q_n, \dots, q_1\}$ into  adjacent and non-overlapping blocks of length $20g(n)+1$;  the number of such complete blocks is $\lfloor \frac{n}{20g(n)+1}\rfloor$. There are $2^{20g(n)+1}$ ways to choose which summands in a given block we take, and at least one of them is the desired gap pattern $\mathcal{G}$. Thus the probability that a given decomposition has pattern $\mathcal{G}$ is at least $1/2^{20g(n)+1}$, so the probability that we do not have $\mathcal{G}$ is at most $1 - 1/2^{20g(n)+1}$. Therefore the probability that the pattern occurs at least once is 
\begin{equation}{\rm Pr}(\mbox{gap substring  } \mathcal{G} \mbox{  occurs in the gap string of }m) \ \geq \  1 - \left(1 - 1/2^{20g(n)+1}\right)^{\lfloor{\frac{n}{20g(n)+1}}\rfloor}.\end{equation} To show this tends to 1 we just need to show the subtracted quantity tends to zero, or equivalently that its logarithm tends to $-\infty$; for large $n$ this is \begin{equation}\left\lfloor{\frac{n}{20g+1}}\right\rfloor \log\left(1 - 1/2^{20g(n)+1}\right) \ \le \ -\frac{n}{21 g(n)} \frac{1/2}{2^{20g(n)}} \ = \ -\frac{2}{21} \frac{n}{g(n) e^{20 g(n) \log(2)}}.\end{equation} If we take $g(n) = C_{\rm FQ} \log(n)$ then \begin{equation}\left\lfloor{\frac{n}{20g+1}}\right\rfloor \log\left(1 - 1/2^{20g(n)+1}\right) \ \le \ -\frac{2}{21 C_{\rm FQ}} \frac{n}{n^{20C_{\rm FQ}\log 2} \log(n)}, \end{equation} which tends to $-\infty$ so long as $C_{\rm FQ} <  1 / 20\log 2$, completing the proof.
\end{proof}

\begin{rek} We could increase the constant $C_{\rm FQ}$ slightly if we replace $2^{20g(n)+1}$ by the number of legal decompositions there are involving the $20g(n)+1$ summands. This is on the order of $q_{20g(n)+1}$; while this is an exponentially growing sequence, it has a smaller base. If we wish to increase the constant by replacing the inequality with an equality we would then have to worry about the logarithm in the denominator. While this could be done at the cost of a more complicated expression, as it is essentially the same size we do not pursue that here.
\end{rek}



\section{Future Research}

We end with a list of additional problems to study for the Fibonacci Quilt; this is a particularly appealing sequence to investigate as it is similar to a PLRS, but is not and has already been shown to have the same behavior for some problems but very different in others. Recall $d(m)$ denotes the number of legal decompositions of $m$ by the Fibonacci quilt.

\begin{itemize}

    \item Can we solve $d(m) = \ell$ for fixed $\ell$? What about $d(m) \le w(m)$ for some fixed increasing function $w$?

    \item How rapidly does $\max_{m \le N} d(m)$ go to infinity?

    \item For $m \le N$, what does the distribution of $d(m)$ look like?

\item Let $K_{\rm min}(m)$ be the fewest number of summands needed in a Fibonacci quilt legal decomposition of $m$ (and similarly define $K_{\rm max}$, $K_{\rm ave}$).  What can we say about $K_{\rm min}$ and $K_{\rm max}$?

        \item Find all $m$ such that $K_{\rm min}(m) = K_{\rm max}(m)$.

\item How does $K_{\rm ave}(m)$ compare to $K_{\rm min}$ and $K_{\rm max}$?  Is is closer to one or the other for all $m$?

\end{itemize}

\appendix

\section{Generating Function Identities for $(s,b)$-Generacci}

Throughout the following we rely on the fact that $(s,b)$-Generacci legal decompositions are unique and thus the largest integer having a legal decomposition using $\{a_1,\ldots,a_n\}$ is less than $a_{n+1}$ by Definition 1.2. See Theorem 1.9 in \cite{CFHMN2} for details.

Let $q_{n,k}$ (with $n,k\geq0$) denote the number of $m\in[0,a_{nb+1})$ whose $(s,b)$-Generacci  legal decomposition contains exactly $k$ summands, where $a_{nb+1}$ is the first entry in the $(n+1)$st bin of size $b$. By definition it is clear that $q_{n,0}=1$ and $q_{n,1}={nb}$  for all $n$. If $n< s+1$ and $k\geq2$, then $q_{n,k}=0$.

\begin{proposition}\label{Fxy}
For $q_{n,k}$ as above, if $n \geq s+1$ and $k\leq\frac{n+s}{s+1}$ then
\begin{equation}
q_{n,k}\ = \ b\cdot q_{n-(s+1),k-1}+q_{n-1,k},\label{qnk_recurrence}
\end{equation}
and if $ k>\frac{n+s}{s+1}$, then $q_{n,k}=0$.

Let $H(x,y)=\sum_{n=0}^{\infty}\sum_{k=0}^{n_{*}}q_{n,k}x^ny^k$ with $n_* = \lceil\frac{n+s}{s+1}\rceil.$ The closed form expression of $H(x,y)$ is
\begin{align}H(x,y)&\ = \ \frac{1+by(x+x^2+\cdots+x^s)}{1-x-byx^{s+1}}.\label{functionH(x,y)}\end{align}
\end{proposition}

\begin{proof} For the first part, note that if $n < 1 +(k-1)(s+1)$, then $q_{n,k}=0$ as we would not have a large enough span to have a $(s,b)$-Generacci legal decomposition with $k$ summands.
Let $\mathcal{B}_i$ denote the $i$\textsuperscript{th} bin in the $(s,b)$-Generacci sequence. If $m\in[0,a_{nb+1})=[0,a_{nb+1}-1]$, then the possible summands come from the set $\{a_1,a_2,a_3,\ldots,a_{nb}\}\subset[0,a_{nb+1})$. If all $k$ summands come from the first $n-1$ bins, that is if they come from the set $\{a_1,a_2,a_3,\ldots,a_{(n-1)b}\},$ then by definition there are $q_{n-1,k}$ many such elements of $[0,a_{(n-1)b+1})$. Now if $m\in[0,a_{nb+1})$ contains a summand from the bin $\mathcal{B}_n$ then the remaining $k-1$ summands must come from the set $\{a_1,a_2,a_3,\ldots,a_{(n-(s+1))b}\}$. This reflects the fact that an $(s,b)$-Generacci  legal decomposition includes summands which must be at least $s+1$ bins away from the entry in the bin $\mathcal{B}_n$. Hence, if $m\in[0,a_{nb+1})$ contains a summands from the bin $\mathcal{B}_n$ (there are $b$ many such possible summands), then there are $b\cdot q_{n-(s+1),k-1}$ many such $m$.
Therefore $ q_{n,k}=b\cdot q_{n-(s+1),k-1}+q_{n-1,k}$. \\ \

For the second part, let $H(x,y)=\sum_{n=0}^{\infty}\sum_{k=0}^{n_{*}}q_{n,k}x^ny^k$.
Then, using \eqref{qnk_recurrence},
\begin{align}
H(x,y)
\ =\ &bx^{s+1}y\displaystyle\sum_{n=s+1}^{\infty}\displaystyle\sum_{k=1}^{n_{*}}q_{n-(s+1),k-1}x^{n-(s+1)}y^{k-1}+x\displaystyle\sum_{n=s+1}^{\infty}\displaystyle\sum_{k=1}^{n_{*}}q_{n-1,k}x^{n-1}y^k\nonumber\\
&\ +\ \displaystyle  \sum_{n=s+1}^\infty q_{n,0}x^n +      \sum_{n=0}^{s}\displaystyle\sum_{k=0}^{n_{*}}q_{n,k}x^ny^k.
\end{align}

Shifting the first sum with $n \leftrightarrow n-(s+1)$ and $k\leftrightarrow k-1$ and the second sum with $n\leftrightarrow n-1$ we have that
\begin{align}
H(x,y)&\ = \ bx^{s+1}y\displaystyle\sum_{n=0}^{\infty}\displaystyle\sum_{k=0}^{n_{*}}q_{n,k}x^{n}y^{k}+x\displaystyle\sum_{n=s}^{\infty}\displaystyle\sum_{k=1}^{n_{*}}q_{n,k}x^{n}y^k+\sum_{n=s+1}^\infty q_{n,0}x^n +\displaystyle\sum_{n=0}^{s}\displaystyle\sum_{k=0}^{n_{*}}q_{n,k}x^ny^k\nonumber\\
&\ = \ bx^{s+1}y\displaystyle\sum_{n=0}^{\infty}\displaystyle\sum_{k=0}^{n_{*}}q_{n,k}x^{n}y^{k}+x\displaystyle\sum_{n=s}^{\infty}\displaystyle\sum_{k=0}^{n_{*}}q_{n,k}x^{n}y^k+\displaystyle\sum_{n=0}^{s}\displaystyle\sum_{k=0}^{n_{*}}q_{n,k}x^ny^k\nonumber\\
&\ = \ bx^{s+1}yH(x,y)+x\left[H(x,y)-\displaystyle\sum_{n=0}^{s-1}\displaystyle\sum_{k=0}^{n_{*}}q_{n,k}x^{n}y^k\right]+\displaystyle\sum_{n=0}^{s}\displaystyle\sum_{k=0}^{n_{*}}q_{n,k}x^ny^k\nonumber\\
&\ = \ bx^{s+1}yH(x,y)+xH(x,y)-x\displaystyle\sum_{n=0}^{s-1}\displaystyle\sum_{k=0}^{n_{*}}q_{n,k}x^{n}y^k+\displaystyle\sum_{n=0}^{s}\displaystyle\sum_{k=0}^{n_{*}}q_{n,k}x^ny^k.
\end{align}
Thus
\begin{equation}
\label{F-equation}
H(x,y)\ = \ bx^{s+1}yH(x,y)+xH(x,y)+h(x,y),
\end{equation}
where \begin{equation}h(x,y)\ = \ -x\displaystyle\sum_{n=0}^{s-1}\displaystyle\sum_{k=0}^{n_{*}}q_{n,k}x^{n}y^k+\displaystyle\sum_{n=0}^{s}\displaystyle\sum_{k=0}^{n_{*}}q_{n,k}x^ny^k.\end{equation}
Solving for $H(x,y)$ in Equation \eqref{F-equation} gives
\begin{equation}H(x,y)\ = \ \frac{h(x,y)}{1-x-byx^{s+1}}.\end{equation} To complete the proof it suffices to show that $h(x,y)=1+by(x+x^2+\cdots+x^s)$.
First observe that
\begin{align}
h(x,y)&\ = \ -x\displaystyle\sum_{n=0}^{s-1}\displaystyle\sum_{k=0}^{n_{*}}q_{n,k}x^{n}y^k+\displaystyle\sum_{n=0}^{s-1}\displaystyle\sum_{k=0}^{n_{*}}q_{n,k}x^ny^k+\displaystyle\sum_{k=0}^{n_*}q_{s,k}x^sy^k\nonumber\\
&\ = \ (1-x)\displaystyle\sum_{n=0}^{s-1}\displaystyle\sum_{k=0}^{n_{*}}q_{n,k}x^{n}y^k+x^s\displaystyle\sum_{k=0}^{n_*}q_{s,k}y^k.\label{eq.h.1}
\end{align}
Recall that $q_{s,k} = 0$ for $k \geq 2$. Also,   if $0\leq n\leq s$, then $q_{n,k}=\begin{cases}1&\mbox{if $k=0$}\\nb&\mbox{if $k=1$}\\0&\mbox{if $k>1$.}\end{cases}$  \\
Hence
\begin{equation}
\label{eq.h.2}
\displaystyle\sum_{k=0}^{n_*}q_{s,k}y^k \ = \  q_{s,0}+q_{s,1}y \ = \  1 + sby.
\end{equation}
Also
\begin{align}
\displaystyle\sum_{n=0}^{s-1}\displaystyle\sum_{k=0}^{n_{*}}q_{n,k}x^ny^k \ = \ \displaystyle\sum_{n=0}^{s-1}(q_{n,0}x^n+q_{n,1}x^ny) \ = \ \displaystyle\sum_{n=0}^{s-1}x^n+yb\displaystyle\sum_{n=0}^{s-1}nx^n. \label{eq.h.3}
\end{align}
Using
\begin{equation}
\displaystyle\sum_{n=0}^{s-1}x^n\ = \ \frac{1-x^{s}}{1-x}  \mbox{ \ and \ } \displaystyle\sum_{n=0}^{s-1}nx^n \ = \  x\displaystyle \frac{d}{dx}\left(\frac{1-x^{s}}{1-x} \right),
\end{equation}
Equation \eqref{eq.h.3} now yields
%
\begin{align}
\displaystyle\sum_{n=0}^{s-1}\displaystyle\sum_{k=0}^{n_*}q_{n,k}x^ny^k&\ = \ \frac{1-x^{s}}{1-x}+yb\left(\frac{x-sx^{s}+(s-1)x^{s+1}}{(1-x)^2}\right)\nonumber\\
&\ = \ \frac{1+(by-1)x-(bsy+1)x^{s}+(yb(s-1)+1)x^{s+1}}{(1-x)^2}.\label{eq2}
\end{align}

Finally, substituting  Equations \eqref{eq.h.2}  and \eqref{eq2} into Equation \eqref{eq.h.1}  gives
\begin{align}
h(x,y)&\ = \ (1-x)\left(\frac{1+(by-1)x-(bsy+1)x^{s}+(yb(s-1)+1)x^{s+1}}{(1-x)^2}\right)+(1+sby)x^s\nonumber\\
&\ = \ \frac{1+(by-1)x-(bsy+1)x^{s}+(yb(s-1)+1)x^{s+1}+(1+sby)x^s(1-x)}{1-x}\nonumber\\
&\ = \ [(1-x)+byx(1-x^s)]/(1-x)\nonumber\\
&\ = \ [(1-x)+byx(1-x)(1+x+\cdots+x^{s-1})]/(1-x)\nonumber\\
&\ = \ 1+by(x+x^2+\cdots+x^s),
\end{align}
which completes the proof.
\end{proof}



\begin{proof}[Proof of Theorem \ref{recurrencefromhell}] We first prove the formula for $p_{n,k}$. Let $a_i=0$ whenever $i\leq 0$. Note that $p_{0,0}=1$ since the interval $[a_{-b+1},a_1)=[0,a_1) = \{0\}$, and zero is the only element that can be written with zero summands.  Similarly, we have $p_{0,k}=0$ whenever $k\geq 1$. For any $n\geq 1$, $p_{n,0}=0$ as no element in the interval $[a_{(n-1)b+1},a_{nb+1})$ can be written with zero summands.

For any $1\leq n\leq s$, $p_{n,1}=b$ as the only elements in the interval $[a_{(n-1)b+1},a_{nb+1})$ with exactly one summand are the entries in the $n$th bin. If $1\leq n\leq s$ and $k\geq 2$, then $p_{n,k}=0$ as there are not enough bins from which we could legally select two summands to decompose elements of the interval $[a_{(n-1)b+1},a_{nb+1})$.

If $n\geq s+1$ and $1\leq k\leq \frac{n+s}{s+1}$, then  any $m$ counted by $p_{n,k}$ will require the decomposition to contain a summand from the $n$th bin and there are $b$ choices for this summand. So the number of integers which can now be created with $k-1$ summands coming from bins $\B_1, \B_2,\ldots, \B_{n-(s+1)}$ is exactly the value of $q_{n-(s+1),k-1}$. Therefore, $p_{n,k}=b\cdot q_{n-(s+1),k-1}$.

To complete the proof we note that if $n\geq s+1$ and $k> \frac{n+s}{s+1}$, then $p_{n,k}=0$ as there are not enough bins from which we could legally select $k$ summands to decompose elements of the interval $[a_{(n-1)b+1},a_{nb+1})$. \\ \

%

Notice that by Equations \eqref{qnk_recurrence} and \eqref{FQpnk} we have that if $n\geq s+1$ and $1\leq k\leq \frac{n+s}{s+1}$, then
\begin{align}
q_{n,k}\ = \ p_{n,k}+q_{n-1,k}.
\end{align}
Hence
\begin{align}
p_{n,k}&\ = \ q_{n,k}-q_{n-1,k}.\label{switchpwithq}
\end{align}

Armed with the above, we can easily finish the proof. Let $F(x,y)=\sum_{n=0}^{\infty}\sum_{k=0}^{n_{*}}p_{n,k}x^ny^k$.
Then, using Equation \eqref{switchpwithq},
\begin{align}
F(x,y)&\ = \ \displaystyle\sum_{n=s+1}^{\infty}\displaystyle\sum_{k=1}^{n_{*}}(q_{n,k}-q_{n-1,k})x^ny^k+\displaystyle\sum_{n=0}^{s}\sum_{k=0}^{n_*}p_{n,k}x^ny^k+\displaystyle\sum_{n=s+1}^{\infty}p_{n,0}x^n.\end{align}
By Equation \eqref{FQpnk}
we note that $p_{n,0}=0$ whenever $n\geq 1$ and $p_{n,k}=0$ whenever $0\leq n\leq s$ and $k\geq 2$. Therefore
\begin{align}
F(x,y)&\ = \ \displaystyle\sum_{n=s+1}^{\infty}\displaystyle\sum_{k=1}^{n_{*}}(q_{n,k}-q_{n-1,k})x^ny^k+p_{0,0}+ \displaystyle\sum_{n=0}^{s}p_{n,1}x^ny\nonumber\\
& \ = \  \displaystyle\sum_{n=s+1}^{\infty}\displaystyle\sum_{k=1}^{n_{*}}q_{n,k}x^ny^k-x\displaystyle\sum_{n=s+1}^{\infty} \displaystyle\sum_{k=1}^{n_{*}}q_{n-1,k}x^{n-1}y^k+1+by\displaystyle\sum_{n=1}^{s}x^n.\label{shiftneeded}
\end{align}
Shifting the second sum in Equation \eqref{shiftneeded} with $n \leftrightarrow n-1$ we have that
\begin{align}
F(x,y)&\ = \ \displaystyle\sum_{n=s+1}^{\infty}\displaystyle\sum_{k=1}^{n_{*}}q_{n,k}x^ny^k-x\displaystyle\sum_{n=s}^{\infty}\displaystyle\sum_{k=1}^{n_{*}}q_{n,k}x^ny^k+1+by \sum_{n=1}^{s}x^n\nonumber\\ & \ = \  (1-x)\left(\displaystyle\sum_{n=s+1}^{\infty}\sum_{k=1}^{n_*}q_{n,k}x^ny^k\right) +1+by\sum_{n=1}^{s}x^n -sbyx^{s+1}.
\end{align}
Continuing \small
\begin{align}
F(x,y)&\ = \ (1-x)\left(\displaystyle\sum_{n=0}^{\infty}\sum_{k=0}^{n_*}q_{n,k}x^ny^k-\displaystyle\sum_{n=0}^{s}\sum_{k=1}^{n_*}q_{n,k}x^ny^k-\sum_{n=0}^\infty q_{n,0}x^n\right)+1+by\sum_{n=1}^{s}x^n-sbyx^{s+1}\nonumber\\
&\ = \ (1-x)\left(H(x,y)-\displaystyle\sum_{n=0}^{s}q_{n,1}x^ny-\sum_{n=0}^\infty q_{n,0}x^n\right)+1+by\sum_{n=1}^{s}x^n-sbyx^{s+1}\nonumber\\
&\ = \ (1-x)\left(H(x,y)-by\sum_{n=0}^{s}nx^n-\sum_{n=0}^\infty x^n\right)+1+by\sum_{n=1}^{s}x^n-sbyx^{s+1}\nonumber\\
&\ = \ (1-x)\left(H(x,y)-byx\frac{d}{dx}\left(\frac{1-x^{s+1}}{1-x}\right)   -\frac{1}{1-x}\right)+1+by\sum_{n=1}^{s}x^n-sbyx^{s+1}\nonumber\\
&\ = \ (1-x)\left(H(x,y)-byx\left(\frac{(1-x)(-(s+1)x^s)+(1-x^{s+1})}{(1-x)^2}\right)-\frac{1}{1-x}\right)+1+by\sum_{n=1}^{s}x^n-sbyx^{s+1}\nonumber\\
&\ = \ (1-x)\left(H(x,y)-byx\left(\frac{-(s+1)x^s+sx^{s+1}+1}{(1-x)^2}\right)-\frac{1}{1-x}\right)+1+by\sum_{n=1}^{s}x^n-sbyx^{s+1}\nonumber\\
&\ = \ (1-x)\left(H(x,y)+\frac{byx((s+1)x^s-sx^{s+1}-1)+(x-1)}{(1-x)^2}\right)+1+by\sum_{n=1}^{s}x^n-sbyx^{s+1}\nonumber\\
&\ = \ (1-x)H(x,y)+\frac{byx((s+1)x^s-sx^{s+1}-1)+(x-1)}{1-x}+\frac{1-x}{1-x}+\frac{by(x-x^{s+1})}{1-x}-\frac{sbyx^{s+1}(1-x)}{1-x} \nonumber\\
&\ = \ (1-x)H(x,y)+\frac{sbyx^{s+1} + byx^{s+1} - sbyx^{s+2} - byx + byx - byx^{s+1} - sbyx^{s+1}+sbyx^{s+2}}{1-x}\nonumber\\
&\ = \ (1-x)H(x,y).
\label{subintothis}\end{align} \normalsize

Recall that Proposition \ref{Fxy} gives
\begin{align}
H(x,y)&\ = \ \displaystyle\sum_{n=0}^{\infty}\displaystyle\sum_{k=0}^{n_{*}}q_{n,k}x^ny^k\ = \ \frac{1+by(x+x^2+\cdots+x^s)}{1-x-byx^{s+1}},
\end{align}
and substituting this into Equation \eqref{subintothis} yields
\begin{align}
F(x,y)&\ =(1-x)\frac{1+by(x+x^2+\cdots+x^s)}{1-x-byx^{s+1}}\nonumber \\
& \ = \ 1 + \frac{byx}{1-x-byx^{s+1}}.
\end{align}
\end{proof}

\section{Generating Function for Greedy-6 FQ-Legal Decompositions}\label{sec:gaussianFQ}

We now return to the FQ sequence and a specific generating function associated to them.
Let $q_{n,k}$ (with $n,k\geq 0$) denote the number of
$m\in[0,q_{n+1})$ whose Greedy-6 FQ-legal decomposition contains exactly $k$ summands.

\begin{proposition}\label{HxyforFQ} Let $n,k\geq 0$. Then
\begin{align}
q_{n,k}\ = \
\begin{cases}
1&\mbox{{\rm if $k=0$ and $n\geq 0$, or if $k=2$ and $n=5$}}\\
n&\mbox{{\rm if $k=1$ and $n\geq 0$}}\\
1+\frac{(n-5)(n-4)}{2}&\mbox{{\rm if $k=2$ and $n\geq 6$}}\\
q_{n-5,k-1}+q_{n-1,k}&\mbox{{\rm if $k\geq 3$ and $n\geq 5(k-1)$}}\\
0&\mbox{{\rm if $k\geq 3$ and $n<5(k-1)$ or if $k=2$ and $n\leq 4$}},
\end{cases}
\end{align} and if $H(x,y)=\sum_{n\geq0}\sum_{k\geq0}q_{n,k}\ x^ny^k$ then
\begin{align}H(x,y)&\ = \ \frac{1+(x+x^2+x^3+x^4)y+x^5y^2}{1-x-yx^5}.\label{FQfunctionH(x,y)}\end{align}
\end{proposition}

\begin{proof}
As the arguments follow from analogous computations as those presented in the proof of Proposition \ref{recurrencefromhell}, we only show the last two cases.

If $k\geq 2$ and $n\leq 4$ a simple observation shows that there are no $m\in[0,q_{n+1})$ which contain two summands in its Greedy-6 FQ-legal decomposition. Now suppose that $k\geq 3$ and $n$ $<$ $5(k-1)$ and let $m\in[0,q_{n+1})$ have $k$ summands in its Greedy-6 FQ-legal decomposition. Then by Theorem~\ref{thm:greedy6} the decomposition of $m$ is given by
\begin{align}
m&\ = \ q_{\ell_1}+q_{\ell+2}+\cdots+q_{\ell_{k-1}}+q_{\ell_k},
\end{align}
where $\ell_1>\ell_2>\cdots>\ell_{k-1}>\ell_k$ and satisfies one of the following conditions:
\begin{enumerate}
\item $\ell_i-\ell_{i+1}\geq 5$ for all $i$, or
\item $\ell_i-\ell_{i+1}\geq 5$ for all $i \leq k-3$, $\ell_{k-2}\geq 10$, $\ell_{k-1}=4$, and $\ell_{k}=2$.
\end{enumerate}
Notice if Condition (1) holds then
\begin{align}
n&\ \geq\ \ell_1\ \geq\ \ell_1-\ell_k \ = \ \sum_{i=1}^{k-1}\ell_{i}-\ell_{i+1}\ \geq\ \sum_{i=1}^{k-1}5  \ = \ 5(k-1).\end{align}
So if $n< 5(k-1)$, then $q_{n,k}=0$.\\

If Condition (2) holds then $\ell_{k-2}\geq 10$, and since $\ell_i-\ell_{i+1}\geq 5$ for all $i\leq k-3$ we see that $\ell_{k-j}\geq 5j$ for all $2\leq j\leq k-1$. Hence $\ell_1\geq 5(k-1)$. Since $\ell_1\leq n$, we have that $n\geq \ell_{1}\geq 5(k-1)$. So again if $n< 5(k-1)$, then $q_{n,k}=0$.

Now suppose that $k\geq 3$ and $n\geq 5(k-1)$. Notice that the arguments above show us that in this case $q_{n,k}\neq 0$. Now note that if $m\in I_n=[0,q_{n+1})$ contains $q_n$ as a summand, then the next possible summand using the Greedy-6 decomposition is $q_{n-5}$. So there are $q_{n-5,k-1}$ integers in $I_n$ which have $k$ summands and include $q_n$ as a summand. If $m\in I_n$ does not contain $q_n$, then the next possible summand is $q_{n-1}$. So there are $q_{n-1,k}$ integers in $ I_n$ which have $k$ summands and do not include $q_n$ as a summand. Thus $q_{n,k}=q_{n-5,k-1}+q_{n-1,k}$.\\ \

We now turn to the second part of the proof. Let $H(x,y)=\sum_{n\geq0}\sum_{k\geq0}q_{n,k}\ x^ny^k$. Then
\begin{align}
\displaystyle\sum_{n\geq0}\displaystyle\sum_{k\geq0}q_{n,k}\ x^ny^k&\ = \ S_1+S_2+S_3+S_4+S_5+S_6+S_7,
\end{align}
where
\begin{align}
S_1\ = \ &\displaystyle\sum_{n\geq 5(k-1)}\displaystyle\sum_{k\geq3}q_{n,k}\ x^ny^k\nonumber\\
S_2\ = \ &\displaystyle\sum_{n<5(k-1)}\displaystyle\sum_{k\geq 3}q_{n,k}\ x^ny^k\nonumber\\
S_3\ = \ &\displaystyle\sum_{n\geq 6}q_{n,2}\ x^ny^2\nonumber\\
S_4\ = \ &\displaystyle\sum_{4\leq n\leq 5}q_{n,2}\ x^ny^2\nonumber\\
S_5\ = \ &\displaystyle\sum_{n\leq 3}q_{n,2}\ x^ny^2\nonumber\\
S_6\ = \ &\displaystyle\sum_{n\geq 0}q_{n,1}\ x^ny\nonumber\\
S_7\ = \ &\displaystyle\sum_{n\geq 0}q_{n,0}\ x^n.
\end{align}
Using Proposition~\ref{HxyforFQ}, we have that
\begin{align}
S_2&\ = \ S_5\ = \ 0\nonumber\\
S_1&\ = \ \displaystyle\sum_{n\geq 5(k-1)}\displaystyle\sum_{k\geq3}(q_{n-5,k-1}+q_{n-1,k})\ x^ny^k\nonumber\\
S_3&\ = \ y^2\displaystyle\sum_{n\geq 6}\left(1+\frac{(n-5)(n-4)}{2}\right)\ x^n\ = \ y^2x^6\displaystyle\sum_{n\geq 0}x^n+\frac{y^2x^6}{2}\displaystyle\sum_{n\geq 2}(n-1)nx^{n-2}\nonumber\\
&\ = \ \frac{y^2x^6}{1-x}+\frac{y^2x^6}{2}\frac{d^2}{dx^2}\left(\frac{1}{1-x}\right)\ = \ \frac{y^2x^6}{1-x}+\frac{y^2x^6}{(1-x)^3}\nonumber\\
S_4&\ = \ x^5y^2\nonumber\\
S_6&\ = \ y\displaystyle\sum_{n\geq 0}n\ x^n\ = \ xy\displaystyle\sum_{n\geq 1}n\ x^{n-1}\ = \ xy\frac{d}{dx}\left(\frac{1}{1-x}\right)\ = \ \frac{xy}{(1-x)^2}\nonumber\\
S_7&\ = \ \displaystyle\sum_{n\geq 0}x^n\ = \ \frac{1}{1-x}.
\end{align}
Now notice that
\begin{align}
S_1&\ = \ \displaystyle\sum_{n\geq 5(k-1)}\displaystyle\sum_{k\geq3}q_{n-5,k-1}\ x^ny^k+\displaystyle\sum_{n\geq 5(k-1)}\displaystyle\sum_{k\geq3}q_{n-1,k}\ x^ny^k\nonumber\\
&\ = \ x^5y\displaystyle\sum_{n-5\geq 5(k-2)}\displaystyle\sum_{k-1\geq2}q_{n-5,k-1}\ x^{n-5}y^{k-1}+x\displaystyle\sum_{n-1\geq 5(k-1)-1}\displaystyle\sum_{k\geq3}q_{n-1,k}\ x^{n-1}y^k\nonumber\\
&\ = \ x^5y\displaystyle\sum_{n\geq 5(k-2)}\displaystyle\sum_{k\geq2}q_{n,k}\ x^{n}y^{k}+x\displaystyle\sum_{n\geq 5(k-1)-1}\displaystyle\sum_{k\geq3}q_{n,k}\ x^{n}y^k\nonumber\\
&\ = \ x^5yA+xB,
\end{align}
where $A=\displaystyle\sum_{n\geq 5(k-2)}\displaystyle\sum_{k\geq2}q_{n,k}\ x^{n}y^{k}$ and $B=\displaystyle\sum_{n\geq 5(k-1)-1}\displaystyle\sum_{k\geq3}q_{n,k}\ x^{n}y^k$.

We have \begin{align}
H(x,y)
&\ = \ A
+\displaystyle\sum_{n\geq5(k-2)}\displaystyle\sum_{k< 2}q_{n,k}x^ny^k
+\displaystyle\sum_{n<5(k-2)}\displaystyle\sum_{k< 2}q_{n,k}x^ny^k
+\displaystyle\sum_{n<5(k-2)}\displaystyle\sum_{k\geq 2}q_{n,k}x^ny^k,
\end{align} and
\begin{align}
\displaystyle\sum_{n\geq5(k-2)}\displaystyle\sum_{k< 2}q_{n,k}x^ny^k
&\ = \ \displaystyle\sum_{n\geq -10}q_{n,0}x^n+y\displaystyle\sum_{n\geq -5}q_{n,1}x^n\nonumber\\
&\ = \ \displaystyle\sum_{n\geq 0}q_{n,0}x^n+y\displaystyle\sum_{n\geq 0}q_{n,1}x^n\nonumber\\
&\ = \ \displaystyle\sum_{n\geq 0}x^n+y\displaystyle\sum_{n\geq 0}nx^n\nonumber\\
&\ = \ \frac{1}{1-x}+\frac{xy}{(1-x)^2},
\end{align} and
\begin{align}
\displaystyle\sum_{n<5(k-2)}\displaystyle\sum_{k< 2}q_{n,k}x^ny^k
&\ = \ \displaystyle\sum_{n<-10}q_{n,0}x^n+\displaystyle\sum_{n<-5}q_{n,1}x^ny\ = \ 0,
\end{align}
and finally
\begin{align}
\displaystyle\sum_{n<5(k-2)}\displaystyle\sum_{k\geq 2}q_{n,k}x^ny^k
&\ = \ \displaystyle\sum_{n<0}q_{n,2}x^ny^2+\displaystyle\sum_{n<5(k-2)}\displaystyle\sum_{k\geq 3}q_{n,k}x^ny^k\ = \ 0
\end{align}
since $q_{n,2}=0$ whenever $n<0$ and also when $n<5(k-2)<5(k-1)$ and $k\geq 3$ we have $q_{n,k}=0$, thus$\displaystyle\sum_{n<5(k-2)}\displaystyle\sum_{k\geq 3}q_{n,k}x^ny^k=0$.

Therefore
\begin{align}
A&\ = \ H(x,y)-\left(\frac{1}{1-x}+\frac{xy}{(1-x)^2}\right).
\end{align}

Now observe that
\begin{align}
H(x,y)&\ = \ B
+\displaystyle\sum_{n\geq5(k-1)-1}\displaystyle\sum_{k<3}q_{n,k}x^ny^k
+\displaystyle\sum_{n<5(k-1)-1}\displaystyle\sum_{k<3}q_{n,k}x^ny^k
+\displaystyle\sum_{n<5(k-1)-1}\displaystyle\sum_{k\geq3}q_{n,k}x^ny^k,
\end{align} and
\begin{align}
\displaystyle\sum_{n\geq5(k-1)-1}\displaystyle\sum_{k<3}q_{n,k}x^ny^k
&\ = \ \displaystyle\sum_{n\geq 0}q_{n,0}x^n+y\displaystyle\sum_{n\geq 0}q_{n,1}x^n+y^2\displaystyle\sum_{n\geq 4}q_{n,2}x^n\nonumber\\
&\ = \ \displaystyle\sum_{n\geq 0}x^n+y\displaystyle\sum_{n\geq 0}nx^n+y^2\left[x^5+\displaystyle\sum_{n\geq6}\left(1+\frac{(n-5)(n-4)}{2}\right)x^n\right]\nonumber\\
&\ = \ \frac{1}{1-x}+\frac{xy}{(1-x)^2}+y^2x^5      +\frac{y^2x^6}{1-x}+\frac{y^2x^6}{(1-x)^3}\nonumber\\
&\ = \ \frac{1+y^2x^6}{1-x}+\frac{xy}{(1-x)^2}  +\frac{y^2x^6}{(1-x)^3}+y^2x^5    ,
\end{align}
and
\begin{align}
\displaystyle\sum_{n<5(k-1)-1}\displaystyle\sum_{k<3}q_{n,k}x^ny^k&\ = \ \displaystyle\sum_{n<-6}q_{n,0}x^n+\displaystyle\sum_{n<-1}q_{n,1}x^ny+\displaystyle\sum_{n<4}q_{n,2}x^ny^2\ = \ 0
\end{align}
and finally
\begin{align}
\displaystyle\sum_{n<5(k-1)-1}\displaystyle\sum_{k\geq3}q_{n,k}x^ny^k&\ = \ 0
\end{align}
since $q_{n,k}=0$ whenever $n<5(k-1)-1<5(k-1)$ and $k\geq 3$. Therefore
\begin{align}
B&\ = \ H(x,y)-\left(\frac{1+y^2x^6}{1-x}+\frac{xy}{(1-x)^2}  +\frac{y^2x^6}{(1-x)^3}+y^2x^5   \right).
\end{align}
Thus
\begin{align}
S_1&\ = \ x^5y\left[H(x,y)-\left(\frac{1}{1-x}+\frac{xy}{(1-x)^2}\right)\right]\nonumber\\
&\hspace{.5in}+x\left[H(x,y)-\left(\frac{1+y^2x^6}{1-x}+\frac{xy}{(1-x)^2}  +\frac{y^2x^6}{(1-x)^3}+y^2x^5   \right)\right],
\end{align}
and
\begin{align}
H(x,y)&\ = \ S_1+S_2+S_3+S_4+S_5+S_6+S_7\nonumber\\
&\ = \ x^5y\left[H(x,y)-\left(\frac{1}{1-x}+\frac{xy}{(1-x)^2}\right)\right]\nonumber\\
&\hspace{.5in}+x\left[H(x,y)-\left(\frac{1+y^2x^6}{1-x}+\frac{xy}{(1-x)^2}  +\frac{y^2x^6}{(1-x)^3}+y^2x^5   \right)\right]\nonumber\\
&\hspace{.5in}
+\frac{y^2x^6}{1-x}+\frac{y^2x^6}{(1-x)^3}
+ x^5y^2
+\frac{xy}{(1-x)^2}
+\frac{1}{1-x},
\end{align}
which imply that
\begin{align}
H(x,y)(1-x-yx^5)&\ = \ -yx^5\left(\frac{1}{1-x}+\frac{xy}{(1-x)^2}\right)\nonumber\\
&\hspace{.2in}-x\left(\frac{1+y^2x^6}{1-x}+\frac{xy}{(1-x)^2}  +\frac{y^2x^6}{(1-x)^3}+y^2x^5   \right)\nonumber\\
&\hspace{.2in}+\frac{y^2x^6}{1-x}+\frac{y^2x^6}{(1-x)^3}
+ x^5y^2
+\frac{xy}{(1-x)^2}
+\frac{1}{1-x}\nonumber\\
&\ = \ 1+y(x+x^2+x^3+x^4)+y^2x^5.
\end{align}

Therefore
\begin{align}
H(x,y)&\ = \ \frac{1+y(x+x^2+x^3+x^4) + y^2x^5}{1-x-yx^5},
\end{align} completing the proof.

\end{proof}

Let $p_{n,k}$ (with $n,k\geq0$) denote the number of $m\in[q_n,q_{n+1})$ whose Greedy-6 FQ-legal decomposition contains exactly $k$ summands. Assuming that $q_0=0$ we have the following result.

\begin{proposition}\label{FxyforFQ}
If $n,k\geq 0$, then
\begin{equation} p_{n,k} \ = \ \begin{cases}
q_{n-5,k-1}&\mbox{{\rm if} $n\geq 5$ {\rm and} $k\geq 1$}\\
1&\mbox{{\rm if} $1\leq n\leq 4$ {\rm and} $k=1$, {\rm or\ if} $n=k=0$}\\
0&\mbox{{\rm otherwise},}\end{cases} \end{equation}
and if $F(x,y)=\sum_{n\geq0}\sum_{k\geq0}p_{n,k}x^ny^k$ then
\begin{align}F(x,y)&\ = \ \frac{1-x+xy-x^5y+x^{10}y^3}{1-x-x^5y}.\label{FQfunctionF(x,y)}\end{align}
\end{proposition}

\begin{proof} We first analyze $p_{n,k}$. The last two cases follow from the definition of $p_{n,k}$ so we focus only on proving that $p_{n,k}=q_{n-5,k-1}$ whenever $n\geq 5$ and $k\geq 1$. Let $m\in[q_n,q_{n+1})$ have exactly $k$ summands in its Greedy-6 FQ-legal decomposition. Then $m$ contains $q_n$ as one of these summands and the largest possible summand of $m-q_n$ is $q_{n+5}$, by definition of the Greedy-6 algorithm. This means that the number of $m\in[q_{n},q_{n+1})$ which have $k$ summands and contain $q_n$ as a summand is the same as the number of $z\in[0,q_{n-5})$ which contain exactly $k-1$ summands. This implies that $p_{n,k}=q_{n-5,k-1}$ as claimed.\\ \

Let $F(x,y)=\sum_{n\geq0}\sum_{k\geq0}p_{n,k}x^ny^k$.
By using Proposition~\ref{FxyforFQ}, we have that
\begin{align}
F(x,y) &\ = \ \displaystyle\sum_{n\geq 5}\displaystyle\sum_{k\geq 1}p_{n,k}x^ny^k+\displaystyle\sum_{n\geq 5}p_{n,0}x^n+\displaystyle\sum_{0\leq n\leq 4}\displaystyle\sum_{k\geq 1}p_{n,k}x^ny^k+\displaystyle\sum_{0\leq n\leq 4}p_{n,0}x^n\nonumber\\
&\ = \ x^5y\displaystyle\sum_{n\geq 5}\displaystyle\sum_{k\geq 1}q_{n-5,k-1}x^{n-5}y^{k-1}+[xy+x^2y+x^3y+x^4y]+1\nonumber\\
&\ = \ x^5y\displaystyle\sum_{n\geq 0}\displaystyle\sum_{k\geq 0}q_{n,k}x^{n}y^{k}+xy+x^2y+x^3y+x^4y+1.
 \end{align}
 By Proposition \ref{HxyforFQ}
\begin{align}
F(x,y)&\ = \ x^5yH(x,y)+xy+x^2y+x^3y+x^4y+1\nonumber\\
&\ = \ x^5y\left(\frac{1+y(x+x^2+x^3+x^4)+y^2x^5}{1-x-yx^5}\right)+xy+x^2y+x^3y+x^4y+1.\nonumber\\
&\ = \ \frac{1-x+xy-x^5y+x^{10}y^3}{1-x-x^5y}.
\end{align}
\end{proof}




\section{An Alternate Approach}

Let  $p_{n,k}$  (with $n,k\geq0$) denote the number of integers in $[0,a_{nb+1})$ whose $(s,b)$-Generacci  legal decomposition contains exactly $k$ summands. An explicit formula for the $p_{n,k}$'s  can be given in terms of binomial coefficients. We then use this explicit form to obtain an alternate expression for the $g_n(y)$'s. The arguments here are more elementary than those in \cite{MW1}; there, delicate generating function arguments were needed in order to obtain results valid for a large class of relations (the positive linear recurrences). If the recurrence is particularly simple, it is possible to avoid many of the technical obstructions. It should not be too surprising that such an approach is possible here, as the $(s,b)$-Generacci are a simple generalization of the Fibonacci numbers, and an elementary approach with explicit formulas involving binomial coefficients was available there (see \cite{KKMW}).

\begin{proposition}
\label{Pnk-explicit}
For all $k \geq 1$ and $n \geq 1 +(k-1)(s+1)$, we have \begin{align}p_{n,k} \ = \  b^k {{n-s(k-1)}\choose{k}}.\end{align}
\end{proposition}

\begin{proof}
Suppose $n = 1 + (k-1)(s+1)$.  Then the only decompositions that use exactly $k$ summands come from the bins $\{\mathcal{B}_1, \mathcal{B}_{1+(s+1)}, \mathcal{B}_{1+2(s+1)}, \dots,  \mathcal{B}_{1+(k-1)(s+1)}=\mathcal{B}_n\}$.  As each bin has $b$  members, $\mathcal{B}_{1 + i(s+1)} = [a_{1 + i(s+1)b},  a_{2 + i(s+1)b}, \dots, a_{b + i(s+1)b}]$, we have $b^k$ different sums of the form $\sum_{i=0}^{k-1}a_{j_i+i(s+1)b}$ where $j_i \in \{1, 2, \dots, b\}$.  Thus, $p_{n,k} =p_{1 + (k-1)(s+1),k}=b^k$.  Noting that
\begin{align}
b^k {{n-s(k-1)}\choose{k}} &\ = \  b^k {{1 + (k-1)(s+1)-s(k-1)}\choose{k}}
 \ = \  b^k {{k}\choose{k}}
\ = \ b^k,
\end{align} we see the proposition holds for $n = 1 + (k-1)(s+1)$.

Now let $n > 1 + (k-1)(s+1)$. We have
\begin{align}
p_{n,k} & \ = \  bp_{n-(s+1),k-1}+p_{n-1,k}\nonumber\\
& \ = \  b\left[b^{k-1}{{(n-(s+1))-s(k-2)}\choose{k-1}}\right] + b^k{{(n-1)-s(k-1)}\choose{k}}\nonumber\\
&\ = \  b^k\left[{{n-s(k-1)-1}\choose{k-1}} + {{n-s(k-1)-1}\choose{k}}  \right]\nonumber\\
&\ = \  b^k{{n-s(k-1)}\choose{k}}.
\end{align}

Another approach for the proof is to consider the ``Cookie Monster'' approach (more commonly, but less entertainingly, referred to as the Stars and Bars method); see \cite{KKMW} for a detailed discussion of its use for the Fibonacci sequence. We are counting decompositions of the form $\sum_{i=1}^k a_{\ell_i}$ where $a_{\ell_1}$ is any member of bin $\mathcal{B}_{\ell_i}$.  Let's define $x_1 := \ell_1-1$ (the number of bins before $\mathcal{B}_{\ell_1}$).  For $2 \leq i \leq k,$ define $x_i :=\ell_i - \ell_{i-1}-1$ (the number of bins between $\mathcal{B}_{\ell_i} $ and $\mathcal{B}_{\ell_{i-1}}$) and set $x_{k+1} := n-\ell_{k}$ (the number of bins after $\mathcal{B}_{\ell_k}$).  We have
\begin{equation}x_1 + 1 + x_2 + 1 + x_3 + 1 +\dots + x_k + 1 + x_{k+1} \ = \  n.\end{equation}

Now define $y_1:= x_1,$ $y_{k+1}:=x_{k+1}$ and $y_i:=x_i-s$ for $2 \leq i \leq k$.  To have a legal decomposition our bins must be separated by at least $s$ other bins and so each $y_i \geq 0$.  Then we have
\begin{align}
y_1 + y_2 + \dots + y_{k} + y_{k+1} &\ = \  x_1 + (x_2-s) + \dots + (x_k-s) + x_{k+1}
\ = \  n-k-(k-1)s.
\end{align}

The number of $(k+1)$-tuples of non-negative integers whose sum is $n-k-(k-1)s$ is given by the binomial coefficient \begin{equation}{{n-k-(k-1)s+k} \choose {k}} \ = \ {{n-(k-1)s} \choose {k}}. \end{equation}

With the chosen bins, we can select any of the $b$ members within each bin, so we have $b^k$ sums.  Thus \begin{equation}p_{n,k} \ = \  b^k {{n-s(k-1)}\choose{k}}.\end{equation}
\end{proof}

Using  Proposition \ref{Pnk-explicit},  we get the following alternate expression for $g_n(y)$.

\begin{proposition}
If  $b\geq 1$, $s\geq 1$ and $n \geq 1 +(k-1)(s+1)$, then
\begin{align}g_n(y) \ = \  \sum_{k=0}^{\lceil\frac{n}{s+1}\rceil}  {{n-s(k-1)}\choose{k}}b^ky^k.
\end{align}

\end{proposition}


\begin{corollary}
The mean and the standard deviation for the number of summands (in the $(s,b)$-Generacci legal decompositions ) for integers in $[0, a_{nb+1})$ are given respectively by
\begin{align}
\mu_n\ = \ \frac{\displaystyle  \sum_{k=1}^{\lceil\frac{n}{s+1}\rceil}  {{n-s(k-1)}\choose{k}}kb^k}{ \displaystyle \sum_{k=0}^{\lceil\frac{n}{s+1}\rceil}  {{n-s(k-1)}\choose{k}}b^k}.
\end{align}

and
\begin{align}
\sigma_n^2\ = \ \frac{\displaystyle  \sum_{k=2}^{\lceil\frac{n}{s+1}\rceil}  {{n-s(k-1)}\choose{k}}k^2b^k}{ \displaystyle \sum_{k=0}^{\lceil\frac{n}{s+1}\rceil}  {{n-s(k-1)}\choose{k}}b^k}+\frac{nb}{ \displaystyle \sum_{k=0}^{\lceil\frac{n}{s+1}\rceil}  {{n-s(k-1)}\choose{k}}b^k} -\left[\frac{\displaystyle  \sum_{k=1}^{\lceil\frac{n}{s+1}\rceil}  {{n-s(k-1)}\choose{k}}kb^k}{ \displaystyle \sum_{k=0}^{\lceil\frac{n}{s+1}\rceil}  {{n-s(k-1)}\choose{k}}b^k}\right]^2.
\end{align}
\end{corollary}


\section{Extended Variance Arguments}\label{sec:EVA}


\subsection{Examples of Sequences}

We give a few examples of sequences which satisfy our assumptions needed to deduce that the mean and variance both grow linearly with $n$ with non-zero leading term. We state what the various blocks and give an example to show how the method works.

\begin{example}[Fibonacci Sequence and Zeckendorf decompositions] \ \\
We have	$\BS=\{ [0], [1, 0]\}$, \;\;$\T=\{ [1]\}$.\\
	An example of a legal decomposition is $F_5 + F_3 + F_1$, and
	its block representation is $[1, 0] [1, 0] [1]$.\\
		After removing the last $\BS$ type block, the new block representation: $[1, 0] [1]$ and
	the resulting legal decomposition: $F_3 + F_1$.	
\end{example}

\begin{example}[A Specific PLRS] \ \\
Consider the PLRS given by	$H_n = 2H_{n-1} + 2H_{n-2} + 0 + 2H_{n-4}$.\\
We have	$\BS =\{[0], [1], [2,0], [2,1], [2 ,2 ,0 ,0], [2,2 ,0 ,1]\}$, \;\; $\T=\{ [2], [2 , 2], [2 ,2  ,0]\}$\\
	An example of a legal decomposition is $H_7 + 2H_4 + H_1$, with block representation $[1] [0] [0] [2, 0] [0] [1]$.\\
		After removing the last $\BS$ type block, the new block representation is $[1] [0] [0] [2, 0] [1]$ and
	the resulting legal decomposition is $H_6 + 2H_3 + H_1$.
\end{example}

\begin{example}[$(1,3)$-Generacci and its decomposition rules]\ \\
We have		$S = \{[0, 0, 0], [1, 0, 0, 0, 0, 0], [0, 1, 0, 0, 0, 0], [0, 0, 1, 0, 0, 0]\}$, $T = \{[1, 0, 0], [0, 1, 0], [0 ,0 ,1]\}$.
\end{example}

\subsection{Proof of Variance Results}\label{sec:blockproofs}

\begin{proof}[Proof of Lemma \ref{lem:collapsingblock}]
	Let $\omega \in \Upsilon_{n,\mathfrak{b}}$ be arbitrary and consider $\phi_{t,\mathfrak{b}}(\omega)$. First, since the block we remove has size $t$ and thus length $l(t)$, $\phi_{t,\mathfrak{b}}(\omega)$ must be in $\Omega_{n-l(t)}$.
	
	Next, consider $\omega, \omega' \in \Upsilon_{n,\mathfrak{b}}$ such that $\phi_{t,\mathfrak{b}}(\omega) = \phi_{t,\mathfrak{b}}(\omega')$.  Since inserting the same block $\mathfrak{b}$ in the same positions to two equal $(\BS,\T)$-legal decompositions leads to the same results, $\phi_{t,\mathfrak{b}}(\omega) = \phi_{t,\mathfrak{b}}(\omega')$ implies $\omega = \omega'$. Thus $\phi_{t,\mathfrak{b}}$ is injective.
	
	Finally, for any $(\BS,\T)$-legal decomposition in $\Omega_{n-l(t)}$, inserting $\mathfrak{b}$, a block of size $t$ and length $l(t)$, after its last $\BS$ type block generates a $(\BS,\T)$-legal decomposition in $\Upsilon_{n,\mathfrak{b}}$. Thus $\phi_{t,\mathfrak{b}}$ is surjective.
	
	Therefore, $\phi_{t,\mathfrak{b}}$ is a bijection between $\Upsilon_{n,\mathfrak{b}}$ and $\Omega_{n-l(t)}$.
\end{proof}

\begin{proof}[Proof of Corollary \ref{cor:var1}]
Because $\phi_{t,\mathfrak{b}}$ is a bijection between $\Upsilon_{n,\mathfrak{b}}$ and $\Omega_{n-l(t)}$, we have
	\begin{eqnarray}
	\E{\nsum|\text{the  last $\BS$ type block is } \mathfrak{b}}& \ =\ &\E{Y_{n-l(t)} + t}\nonumber\\
	& =&  C(n-l(t)) + d + f(n-l(t)) + t.
	\end{eqnarray}
	Hence
	\begin{eqnarray}
	\E{\nsum|Z_n = t} &\ =\ & \sum_{\mathfrak{b}\in \B_t}{\E{\nsum|\text{the last $\BS$ type block is }\mathfrak{b}} \PP{\text{the  last $\BS$ type block is }\mathfrak{b}| Z_n = t}}\nonumber\\
	& =& \E{Y_{n-l(t)} + t} \sum_{\mathfrak{b}\in \B_t}{\PP{\text{the  last $\BS$ type block is } \mathfrak{b} | Z_n = t}}\nonumber\\
	& =& \E{Y_{n-l(t)} + t}\nonumber\\
	& =& C(n-l(t)) + d+ f(n-l(t)) + t,
	\end{eqnarray}
	
	Next, we have
	\begin{align}
				\E{\nsum^2|Z_n = t} &\ =\ \sum_{\mathfrak{b}\in \B_t}{\E{\nsum^2|\text{the  last $\BS$ type block is }\mathfrak{b}} \PP{\text{the  last $\BS$ type block is }\mathfrak{b}| Z_n = t}}\nonumber\\
				&\ = \ \E{(Y_{n-l(t)} + t)^2} \sum_{\mathfrak{b}\in \B_t}{\PP{\text{the  last $\BS$ type block is } \mathfrak{b} | Z_n = t}}\nonumber\\
				&\ = \ \E{(Y_{n-l(t)} + t)^2}\nonumber\\
				& \ = \  \E{Y_{n-l(t)}^2 + 2t{Y_{n-l(t)}} + t^2}\nonumber\\
				& \ = \  \E{Y_{n-l(t)}^2} + 2t\E{Y_{n-l(t)}} + t^2\nonumber\\
				& \ = \  \E{Y_{n-l(t)}^2} + 2t[C(n-l(t)) + d + f(n-l(t))] + t^2.
		\end{align}

Furthermore, by $\eqref{1}$ we have
		\begin{align}
			\E{\nsum} & \ = \  \sum\limits_{t = 0}^{\ZS} \PP{Z_n = t}\cdot \E{\nsum|Z_n = t}\nonumber\\
			& \ = \  Cn + d + \sum\limits_{t  = 0} ^ {\ZS} \PP{Z_n = t}\cdot [t + f(n-l(t)) - Cl(t)]\nonumber\\
			& \ = \  Cn + d + f(n),
		\end{align}
where the last equality comes from the fact that $\E{\nsum} = Cn+d+f(n)$.
Thus
\begin{equation}
	\E{K_n} \  = \  \sum\limits_{t  = 0} ^ {\ZS} \PP{Z_n = t}\cdot [t + f(n-l(t)) - Cl(t)] \ = \ f(n).
\end{equation}
\end{proof}

\begin{proof}[Proof of Corollary \ref{cor:p0}]
Denote the block with size 0 by $\emptyset$. Recall that in $\BS$ a block with size 0 has the shortest length of all blocks. Hence for an arbitrary block $\mathfrak{b}$ with length $t$,
\begin{align}
\PP{\mbox{last $\BS$ type block is } \emptyset} &\ =\ \frac{|\Upsilon_{n,\emptyset}|}{|\Omega_n|}\nonumber\\
&\ = \ \frac{H_{n - l(0) + 1} - H_{n - l(0)}}{H_{n+1} - H_n}\nonumber\\
&\ \geq\ \frac{H_{n - l(t) + 1} - H_{n - l(t)}}{H_{n+1} - H_n}\nonumber\\
&\ =\ \frac{|\Upsilon_{n,\mathfrak{b}}|}{|\Omega_n|}\nonumber\\
&\ =\ \PP{\mbox{last $\BS$ type block is } \mathfrak{b}}.
\end{align}

Since $ \sum_{\mathfrak{b}\in \BS}\PP{\mbox{last $\BS$ type block is } \mathfrak{b}} =1$, we have $\PP{\mbox{last $\BS$ type block is } \emptyset} \geq \frac{1}{|\BS|}$. Since there is only one block with size 0, $\PP{Z_n = 0} \geq 1/|\BS|$.
\end{proof}

\ \\

\end{document}